\documentclass{amsart}                     
\usepackage{graphicx}
\usepackage{latexsym}
\usepackage[T1]{fontenc}
\usepackage[utf8]{inputenc}
\usepackage[english]{}
\usepackage{amsmath}
\usepackage{amssymb}
\usepackage{amsfonts}
\usepackage{times}
\usepackage{dsfont}

\newtheorem{theo}{Theorem}

\newtheorem{coro}{Corollary}
\newtheorem{lemm}{Lemma}
\newtheorem{rema}{Remark}
\newtheorem{prop}{Proposition}

\newtheorem{conj}{Conjecture}

\newcommand{\Z}{{\mathbb Z}}

\newcommand{\R}{{\mathbb R}}

\renewcommand{\P}{{\mathbb P}}
\newcommand{\E}{{\mathbb E}}

\newcommand{\Vol}{\operatorname{Vol}}
\newcommand{\Span}{\operatorname{span}}

\newcommand{\tref}[1]{Theorem~\textup{\ref{#1}}}

\newcommand{\cref}[1]{Corollary~\textup{\ref{#1}}}
\newcommand{\lref}[1]{Lemma~\textup{\ref{#1}}}

\newcommand{\fref}[1]{Figure~\textup{\ref{#1}}}

\begin{document}

\title{On the density of sets avoiding parallelohedron distance $1$}

\thanks{This study has been carried out with financial support from the French State, managed by the French National Research Agency (ANR) in the frame of the "Investments for the future" Programme IdEx Bordeaux - CPU (ANR-10-IDEX-03-02)}



\author{Christine Bachoc \and Thomas Bellitto \and Philippe Moustrou \and Arnaud P\^echer
}


\address{Christine Bachoc, Institut de Math\'ematiques de Bordeaux, UMR 5251, Universit\'e de Bordeaux, 351 cours de la Lib\'eration, 33400 Talence, France.}
\email{christine.bachoc@u-bordeaux.fr}

\address{Thomas Bellitto, LaBRI, Universit\'e de Bordeaux, 351 cours de la Lib\'eration, 33400 Talence, France.}
\email{thomas.bellitto@u-bordeaux.fr}

\address{Philippe Moustrou, Institut de Math\'ematiques de Bordeaux, UMR 5251, Universit\'e de Bordeaux, 351 cours de la Lib\'eration, 33400 Talence, France.}
\email{philippe.moustrou@u-bordeaux.fr}

\address{Arnaud P\^echer, LaBRI, Universit\'e de Bordeaux, 351 cours de la Lib\'eration, 33400 Talence, France.}
\email{arnaud.pecher@u-bordeaux.fr}

\date{\today}


\begin{abstract}
The maximal density of a measurable subset of $\R^n$
  avoiding Euclidean distance $1$ is unknown except in the trivial
  case of dimension $1$. In this paper, we consider the case of a
  distance associated to a polytope that tiles space, where it is
  likely that the  sets avoiding distance $1$ are of maximal density
  $2^{-n}$, as conjectured by Bachoc and Robins. We prove that this is true for $n=2$, and for the
  Vorono\"\i\  regions of the lattices $A_n$, $n\geq 2$.
\end{abstract}

\maketitle
\section{Introduction}

\thispagestyle{empty}

A set \textit{avoiding distance $1$} is a set $A$ in a normed vector space $(\R^n,\Vert  \cdot\Vert  )$ such that $\Vert  x-y\Vert  \neq 1$ for every $x,y\in A$. The number $m_1(\R^n,\Vert\cdot\Vert  )$ measures the highest proportion of space that can be filled by a set avoiding distance $1$. More precisely, $m_1(\R^n,\Vert \cdot\Vert  )$ is the supremum of the \textit{densities} (see Subsection \ref{prelimDP} for a precise definition) of Lebesgue measurable sets $A\subset \R^n$ avoiding distance $1$.

The problem of determining $m_1(\R^n,\Vert \cdot \Vert  )$ has been
mostly studied in the Euclidean case. The number
$m_1(\R^n)=m_1(\R^n,\Vert \cdot \Vert  _2)$ was introduced by Larman
and Rogers in \cite{MR0319055} as a tool to study the
\textit{measurable chromatic number} $\chi_m(\R^n)$ of $\R^n$, which is the minimal number of colors required to color $\R^n$ in such a way that two points at Euclidean distance $1$ have distinct colors, and that the color classes are measurable. 
Determining $\chi_m(\R^n)$ has turned out to be a very difficult problem, that has only been solved in dimension $1$, and that is wide 
open in any other dimension, including the familiar dimension $2$,
where it is only known that $5\leq \chi_m(\R^2)\leq 7$ (see
\cite{Falconer}, \cite{MR1954746}, and \cite[Chapter 3]{Soifer} for
a detailed historical account).

The connection between $m_1(\R^n)$ and $\chi_m(\R^n)$ lies in the following inequality:
$$ \chi_m(\R^n) \geq \frac{1}{m_1(\R^n)}, $$
so, from an upper bound for $m_1(\R^n)$, one obtains a lower bound for $ \chi_m(\R^n)$.

A natural approach to build a set avoiding distance $1$, that works for any norm, starts from a packing of unit balls. Let $\Lambda$ be a set such that if $x,y\in \Lambda$, then the unit open balls $B(x,1)$ and $B(y,1)$ do not overlap. Then the set $A=\cup_{\lambda \in \Lambda} B(\lambda,1/2)$ of disjoint balls of radius 1/2 is a set avoiding $1$ and its density is 
$\frac {\delta} {2^n}$ where $n$ is the dimension of the space and $\delta$ is the density of the packing. This construction is illustrated in \fref{PackBoules}.

\begin{figure}[!ht]
\includegraphics[scale=1.2]{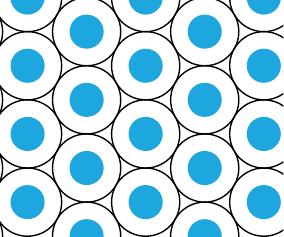}
\caption{A set avoiding distance 1 built from a sphere packing.\label{PackBoules}}
\end{figure}

In the Euclidean plane, the density of an optimal packing of discs of radius $1$ is $0.9069$ and this approach therefore provides a lower bound of $0.9069/4=0.2267$ for $m_1(\R^2,\Vert\cdot\Vert_2)$. 
The best known construction is not much better than that: by refining this idea, Croft manages to build in \cite{Croft} a set of density $0.2293$, which is an arrangement of balls cut out by hexagons.

Regarding upper bounds, Erd{\H o}s conjectured (see \cite{MR1954746}) that
$$m_1(\R^2)<\frac{1}{4}.$$
The best upper bound up to now is due to 
Keleti, Matolcsi, de Oliveira Filho and Ruzsa \cite{Keleti2016}, who have shown $m_1(\R^2)\leq 0.258795$.
Moser, Larman and Rogers (see \cite{MR0319055}) generalized Erd{\H o}s' conjecture to higher dimensions: for every $n \geq 2$,
$$m_1(\R^n)<\frac{1}{2^n}.$$
A weaker result has been proved in \cite{Keleti2016}: a set avoiding distance $1$ necessarily has a density strictly smaller than $\frac 1 {2^n}$ if it has a \textit{block structure}, \textit{i.e.} if it may be decomposed as a disjoint union $A=\cup A_i$ such that if $x$ and $y$ are in the same block $A_i$ then $\Vert   x-y \Vert   <1$ and if they are not, $\Vert   x-y \Vert   >1$. However, without this assumption, the known upper bounds are pretty far from $2^{-n}$, even asymptotically: the best asymptotic bound is $m_1(\R^n)\leq (1+ o(1))(1.2)^{-n}$ (see  \cite{MR0319055}, \cite{MR3341578}).

Going back to the general case of an arbitrary norm, we make the remark that if the unit ball tiles $\R^n$ by translation, the method described previously to build a set avoiding distance 1 from a packing provides a set of density exactly $1/2^n$, as illustrated in Figure \ref{PackHexa}. Moreover, it is likely that this construction of a set avoiding distance 1 is optimal, as conjectured by Bachoc and Robins:
\begin{figure}[!ht]
\includegraphics[scale=0.7]{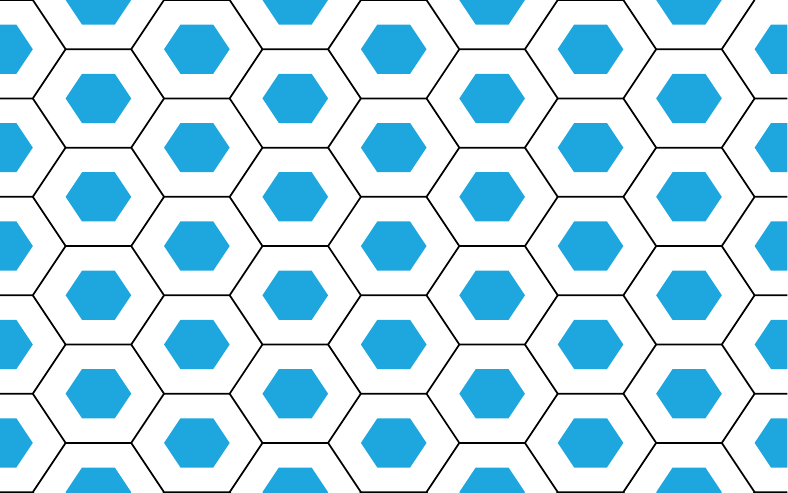}
\caption{The natural construction of density $1/2^n$.\label{PackHexa}}
\end{figure}

\begin{conj}[Bachoc, Robins]
If $\Vert \cdot \Vert$ is a norm such that the unit ball tiles $\R^n$ by translation, then
$$ m_1(\R^n,\Vert \cdot \Vert )= \frac{1}{2^n}. $$
\end{conj}

In this paper, we prove Conjecture 1 in dimension $2$:

\begin{theo}\label{TheoPlan}
If $\Vert \cdot \Vert$ is a norm such that the unit ball tiles $\R^2$ by translation, then
$$ m_1(\R^2,\Vert \cdot \Vert )= \frac{1}{4}. $$
\end{theo}

Recall that the only convex bodies that tile space by translation are the \textit{parallelohedra}, \textit{i.e.} the polytopes that admit a face-to-face tiling by translation. For a given parallelohedron $\mathcal{P}$, we denote by $\Vert \cdot\Vert  _\mathcal{P}$ the norm whose unit ball is  $\mathcal{P}$.

The Vorono\"\i\  region of a lattice is a parallelohedron. Conversely, Vorono\"\i\  conjectured that all parallelohedra are, up to affine transformations, the Vorono\"\i\  regions of lattices (see Subsection \ref{prelimVoro}). On the other hand, $m_1(\R^n, \Vert\cdot\Vert)$ is clearly left 
unchanged under the action of a linear transformation applied to the
norm. So, in the light of Vorono\"\i's conjecture, it is natural to
consider in first place the polytopes that are Vorono\"\i\ regions of lattices.

The most obvious family of  lattices is the  family of cubic lattices $\Z^n$, whose
Vorono\"\i\ regions are hypercubes. We will see that in this case,
Conjecture 1 holds trivially. The next families of lattices to consider
are arguably the root lattices $A_n$ and $D_n$, where
\begin{equation*}
A_n=\{x\in \Z^{n+1} \mid \sum_{i=1}^{n+1} x_i=0\} \quad (n\geq 2).
\end{equation*}
and 
\begin{equation*}
D_n=\{x\in \Z^n\mid \sum_{i=1}^n x_i\equiv 0\bmod
2\} \quad (n\geq 4).
\end{equation*}

We will prove
Conjecture 1 for the Vorono\"\i\ regions of the lattices $A_n$ in every dimensions $n\geq
2$. 
For the lattices $D_n$, 
we can only show the inequality
$$m_1(\R^n,\Vert   \cdot\Vert  _\mathcal{P})\leq \frac{1}{(3/{4})2^n+n-1}$$
which is however asymptotically of the order $O\left(\frac 1 {2^n}\right)$.

Let us now give an idea of the method that we use to prove these results.
The strategy is to transfer the study of sets avoiding distance $1$ to a discrete setting, in which such
sets can be decomposed as the disjoint union of small pieces (in other
words they afford a kind of block structure). Computing the optimal density of a set avoiding distance $1$ in the
discrete setting amounts then to understanding how these blocks fit
together locally. 

To be more precise, we consider  discrete subsets $V$ of $\R^n$,
seen as induced subgraphs of the \emph{unit distance graph}
$G(\R^n,\Vert \cdot \Vert)$. This is the graph whose vertices are the points of $\R^n$
and whose edges connect the vertices $x$ and $y$ if and only if $\Vert
x-y \Vert=1$.

If $G=(V,E)$ is a finite induced subgraph of $G(\R^n,\Vert \cdot
\Vert)$, then it is well known that (see \cite{MR0319055})
$$m_1(\R^n,\Vert \cdot \Vert) \leq \frac{\alpha(G)}{|V|}, $$
where as usual $\alpha(G)$ denotes the \textit{independence number} of
$G$ and $|V|$ is the number of its vertices. We use a generalization
of this inequality to discrete graphs (see Subsection
\ref{prelimGraph}). Of course, the most difficult task is to design an
appropriate discrete subset $V$, i.e. one that provides a good
upper bound  of $m_1(\R^n,\Vert \cdot \Vert)$ and at the same time is
easy to analyse.

For the regular hexagon in the plane, we follow an idea due to  Dmitry Shiryaev \cite{Dmitry} 
who proposed an auxiliary graph satisfying the following remarkable property: if two points $x$ and $y$ are at graph distance $2$, then they are at polytope distance 1. This implies that a set avoiding polytope distance $1$ is a union of cliques whose closed neighborhoods are disjoint. The density of such a set is bounded by the supremum of the \textit{local densities} of the cliques in their closed neighborhood. 
In the case of a general hexagonal Vorono\"\i\  cell in the plane, this approach doesn't work straightforwardly and  we need to introduce a different graph with a slightly weaker property.
The construction of such an auxiliary graph is also a key ingredient of our proofs of the bounds for the Vorono\"\i\  regions of $A_n$ and $D_n$. 

The paper is organized as follows: Section \ref{prelims} contains
preliminaries. In Section \ref{Plane}, we prove
\tref{TheoPlan}. Section \ref{SectionFamilles} is dedicated to the
families of lattices  $A_n$ (Theorem \ref{AnTheo}) and $D_n$ (Theorem
\ref{DnTheo}). In Section \ref{SecChi}, we discuss the chromatic
number of the unit distance graph $G(\R^n, \Vert\cdot\Vert_{\mathcal  P})$. 

We provide in Appendix \ref{AppA}, the rather technical proof of Lemma \ref{alpha1=alpha2}, which gives an alternate definition of the maximal density of an independent set of a discrete graph 
whose vertices have finite degrees. 

\section{Preliminaries}\label{prelims}

\subsection{The density of a set avoiding polytope distance 1}\label{prelimDP}

Let $\R^n$ be equipped with a norm $\Vert   \cdot \Vert  $.  A set $S\subset \R^n$ is said to \textit{avoid 1} if for every $x,y\in S$, $d(x,y)=\Vert   x-y \Vert  \neq 1$. We define the \textit{density} of a measurable set $A\subset \R^n$ with respect to Lebesgue measure as: 

$$\delta(A)=\limsup_{R\to\infty} \frac{\Vol(A\cap [-R,R]^n)}{\Vol([-R,R]^n)},$$
 and we denote by $m_1(\R^n,\Vert   \cdot\Vert  )$ the supremum of the densities achieved by measurable sets avoiding distance 1:

$$m_1(\R^n,\Vert   \cdot\Vert  )=\sup_{ \substack{ S\subset \R^n \text{measurable} \\ S \text{ avoiding } 1}}  \delta(S).$$

Let $\mathcal{P}$ be a convex symmetric polytope. The norm $\Vert   \cdot\Vert  _\mathcal{P}$ associated with $\mathcal{P}$ is defined by
$$\Vert   x\Vert  _\mathcal{P}=\inf\{ \lambda \in \R_+\mid   x \in \lambda \mathcal{P} \},$$
and we call \textit{polytope distance} the distance induced by $\Vert \cdot \Vert_\mathcal{P}$.

If $B_\mathcal{P}(r)=\{ x\in\R^n\mid \Vert  x\Vert  _\mathcal{P}<r\}$, we have by definition:
$$x\in B_\mathcal{P}(1) \Leftrightarrow x \in \mathring{\mathcal{P}}\  \text{and}\ \Vert  x\Vert  _\mathcal{P}=1 \Leftrightarrow x \in \partial\mathcal{P},  $$
where $\mathring{\mathcal{P}}$ denotes the interior of $\mathcal{P}$ and $\partial\mathcal{P}$ its boundary.

A polytope $\mathcal{P}$ \textit{tiles} $\R^n$ \textit{by translations} if there exists $\Lambda\subset \R^n$ such that ${\bigcup_{\lambda \in \Lambda} (\lambda + \mathcal{P}) =\R^n}$ and for every $\lambda \neq \lambda '$,  $(\lambda + \mathring{\mathcal{P}} ) \cap (\lambda' + \mathring{\mathcal{P}}) = \emptyset$. 
If $\mathcal{P}$ is such a polytope, the set 
$$A=\bigcup_{\lambda \in \Lambda} (\lambda + \frac{1}{2}\mathring{\mathcal{P}})$$
avoids 1, and has density $\frac{1}{2^n}$. This set gives a lower bound for $m_1$:

\begin{prop}\label{borneinf}
If $\mathcal{P}$ is a polytope tiling $\R^n$ by translation, and $\Vert   \cdot\Vert  _\mathcal{P}$ the norm associated with $\mathcal{P}$, then 
$$m_1(\R^n,\Vert   \cdot\Vert  _\mathcal{P})\geq \frac{1}{2^n}.$$
\end{prop}

\subsection{Parallelohedra and the Vorono\"\i\ 's Conjecture}\label{prelimVoro}

A $n$-dimensional \textit{parallelohedron} is a polytope $\mathcal{P}$ that tiles \textit{face-to-face} $\R^n$ by translation, i.e there is a tiling such that the intersection between two translates of $\mathcal{P}$, if non empty, is a common face of both of them. Works by Minkowski \cite{Minkowski1897}, Venkov \cite{MR0094790}, and McMullen \cite{MR582003} have led to a proof that the convex bodies tiling space by translation are exactly the parallelohedra, and moreover they tile $\R^n$ by a lattice.

Let us recall that a \textit{lattice} $\Lambda\subset \R^n$ is a discrete subgroup of the form $\bigoplus_{i=1}^n \Z e_i$ where $\{e_1,\ldots,e_n \}$ is a basis of $\R^n$ (for a general reference on lattices, see e.g \cite{Conway:1987:SLG:39091}). The \textit{Vorono\"\i\  region} of $\Lambda$ is defined by
$$\mathcal{V}=\mathcal{V}_\Lambda=\{z\in \R^n, \forall \text{ } x\in  \Lambda, \langle z-x, z-x \rangle \geq \langle z, z \rangle \},$$
where  $\langle \cdot  , \cdot \rangle$ denotes the usual scalar product on $\R^n$.
The Vorono\"\i\  region of a lattice is a parallelohedron. Vorono\"\i\  conjectured that the converse is also true, up to an affine transformation:

\begin{conj}[Vorono\"\i\ 's Conjecture] 
If $\mathcal{P}$ is a parallelohedron in $\R^n$, then there is an affine map $\varphi:\R^n\to\R^n$ such that $\varphi(\mathcal{P})$ is the Vorono\"\i\  region of a lattice $\Lambda\subset \R^n$.
\end{conj}

This conjecture has been solved for several families of parallelohedra. For instance, Vorono\"\i\  himself \cite{MR1580754} proved it for \textit{primitive parallelohedra}, and Erdahl \cite{MR1703597} solved it for \textit{zonotopal parallelohedra}. Moreover, Delone \cite{zbMATH02567417} has shown that Vorono\"\i\ 's conjecture is true in dimensions up to $4$.

According to Vorono\"\i\ 's conjecture, we focus on polytopes that are Vorono\"\i\  regions of lattices.

\subsection{Discretization of the problem}\label{prelimGraph}

A set avoiding distance 1 in $\R^n$ is exactly an independent set in $G(\R^n,\Vert \cdot \Vert)$, \textit{i.e.} a subset $S$ of vertices such that, for all $x,y\in S$, $\Vert x-y \Vert\neq 1$. Therefore $m_1(\R^n,||.||)$ is the supremum of the densities achieved by independent sets. It is the analogue of the \textit{independence ratio} $\bar{\alpha}(G)=\frac{\alpha(G)}{|V|}$ of a finite graph $G$. 

Let $G=(V,E)$ be a discrete induced subgraph of $G(\R^n,\Vert \cdot \Vert)$. For $A\subset V$, we define the density of $A$ in $G$:

\begin{equation}
\delta_G(A)=\limsup_{R\to \infty} \frac{|A\cap V_R|}{|V_R|} \label{DeltaG} 
\end{equation}
 
where $V_R=V\cap [-R,R]^n$. Based on this notion, we extend the definition of the independence ratio to discrete graphs:
$$ \bar{\alpha}(G)=\sup_{A \text{ independent set}} \delta_G(A). $$ 

In this paper, we use the following equivalent formulation of $\bar{\alpha}(G)$:

\begin{lemm}\label{alpha1=alpha2}
Let $G=(V,E)$ be a discrete graph with $V\subset \R^n$. If every $v\in V$ has finite degree, then
$$ \bar{\alpha}(G)=\limsup_{R\to \infty} \bar{\alpha}(G_R),$$
where $G_R$ is the finite induced subgraph of $G$ whose set of vertices is $V_R=V\cap [-R,R]^n$.
\end{lemm}

\begin{proof}
 This lemma is proved in Appendix \ref{AppA} along with a discussion on the importance of the hypothesis that all the vertices of the graph have finite degree.
\end{proof}

Discrete subgraphs induced by $G(\R^n,\Vert \cdot \Vert)$  provide upper bounds of $m_1(\R^n,||.||)$ thanks to the following lemma:

\begin{lemm}\label{LemmeSousGraphe}
Let $G=(V,E)$ be a discrete subgraph induced by $G(\R^n,\Vert \cdot \Vert)$. Then 
\[m_1(\R^n,||.||)\leq \bar{\alpha}(G).\]
\end{lemm}

\begin{proof}
By \lref{alpha1=alpha2}, we may assume without loss of generality that $G$ is finite. In this case the result is well known: the proof below is for the sake of completeness.

Let $R>0$ be a real number, and let $X\in [-R,R]^n$ chosen uniformly at random. For $S\subset \R^n$, the probability that $X$ is in $S$ is $\P(X\in S)= \frac{\Vol(S\cap [-R,R]^n)}{\Vol ([-R,R]^n)}$. Notice that $\limsup_{R\to\infty} \P(X \in S) = \delta(S)$. 

Let $S\subset\R^n$ be a set avoiding $1$. We define the random variable $N=|(X+V)\cap S|$. On one hand, we have:
$$\begin{aligned} \E\left[\frac{N}{|V|}\right]&=\frac{1}{|V|}\E\left[\sum_{v\in V} \mathds{1}_{\{ X + v\in S \}} \right] \\& =\frac{1}{|V|} \sum_{v\in V} \P( X\in S-v). \end{aligned}$$
For every $v$, we have $\limsup_{R\to\infty} \P( X\in S-v)=\delta(S-v)=\delta(S)$. 

On the other hand, since for $v_1,v_2\in V$, ${\Vert  (X-v_1)-(X-v_2)\Vert  =\Vert  v_1-v_2\Vert  }$, and $(X+V)\cap S\subset S$, we have, for any $R>0$, 
$$ \frac{N}{|V|} \leq \bar{\alpha}(G). $$

Thus we get, 
$$\delta(S)\leq \bar{\alpha}(G).$$

\end{proof}

In order to give a first example, we consider the most natural lattice: the cubic lattice. The associated tiling and norm are respectively the cubic tiling and the well known sup norm 
$\Vert x\Vert  _\infty=\sup_{1\leq i\leq n} |x_i|$. More precisely, if $L=2\Z^n$, the Vorono\"\i\  region of $L$ is the cube whose vertices are the points of coordinates $(\pm 1,\pm 1,\ldots,\pm 1)$.

\begin{prop}
For every $n\geq1$, we have:
$$m_1(\R^n, \Vert   \cdot\Vert  _\infty)=\frac{1}{2^n}$$
\end{prop}

\begin{proof}
Let $V=\{0,1\}^n\subset \R^n$ 
  and let $G$ be the subgraph of $G(\R^n,\Vert \cdot \Vert)$ induced by $V$. Following the definition of $V$, for every $v,v'\in V$ with $v\neq v'$, we have $\Vert  v-v'\Vert  _\infty=1$. So $G$ is a complete graph, thus its independence number is $1$. Since it has $2^n$ vertices, applying \lref{LemmeSousGraphe}, we get
$$ m_1(\R^n, \Vert   \cdot\Vert  _\infty) \leq \frac{\alpha(G)}{|V|} =\frac{1}{2^n}.$$ 
\end{proof}

\section{Parallelohedron norms in the plane} \label{Plane}

In this section, we prove Theorem \ref{TheoPlan}. 
It is well known that the  parallelohedra in dimension $2$, are, up  to an affine 
transformation, the Vorono\"\i\  regions of a lattice, and that their combinatorial type is either
that of a square or of a hexagon (see Figure \ref{Voro2}).

\begin{figure}[!ht]
\includegraphics[scale=0.25]{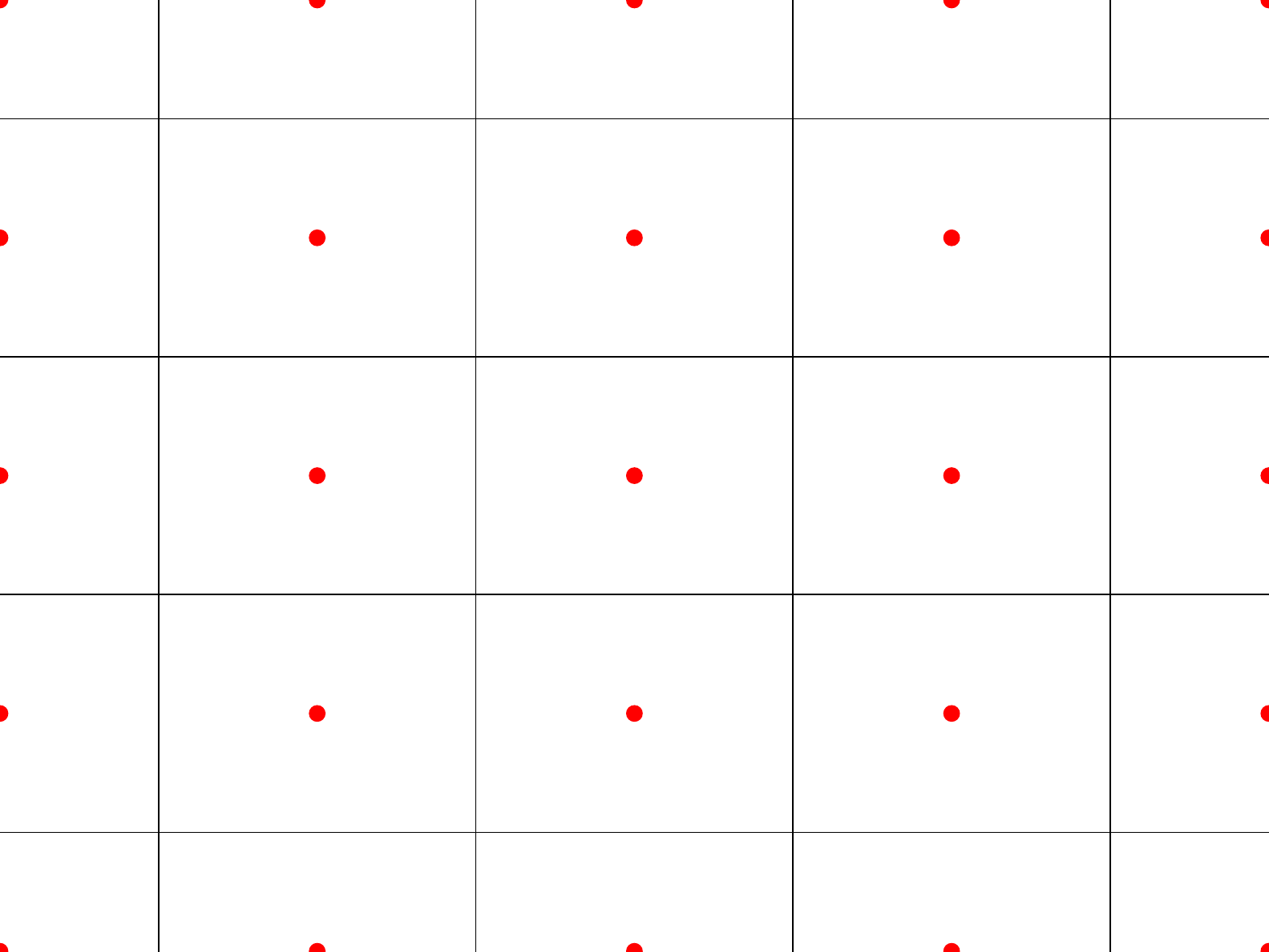}
\hspace{20pt}
\includegraphics[scale=0.25]{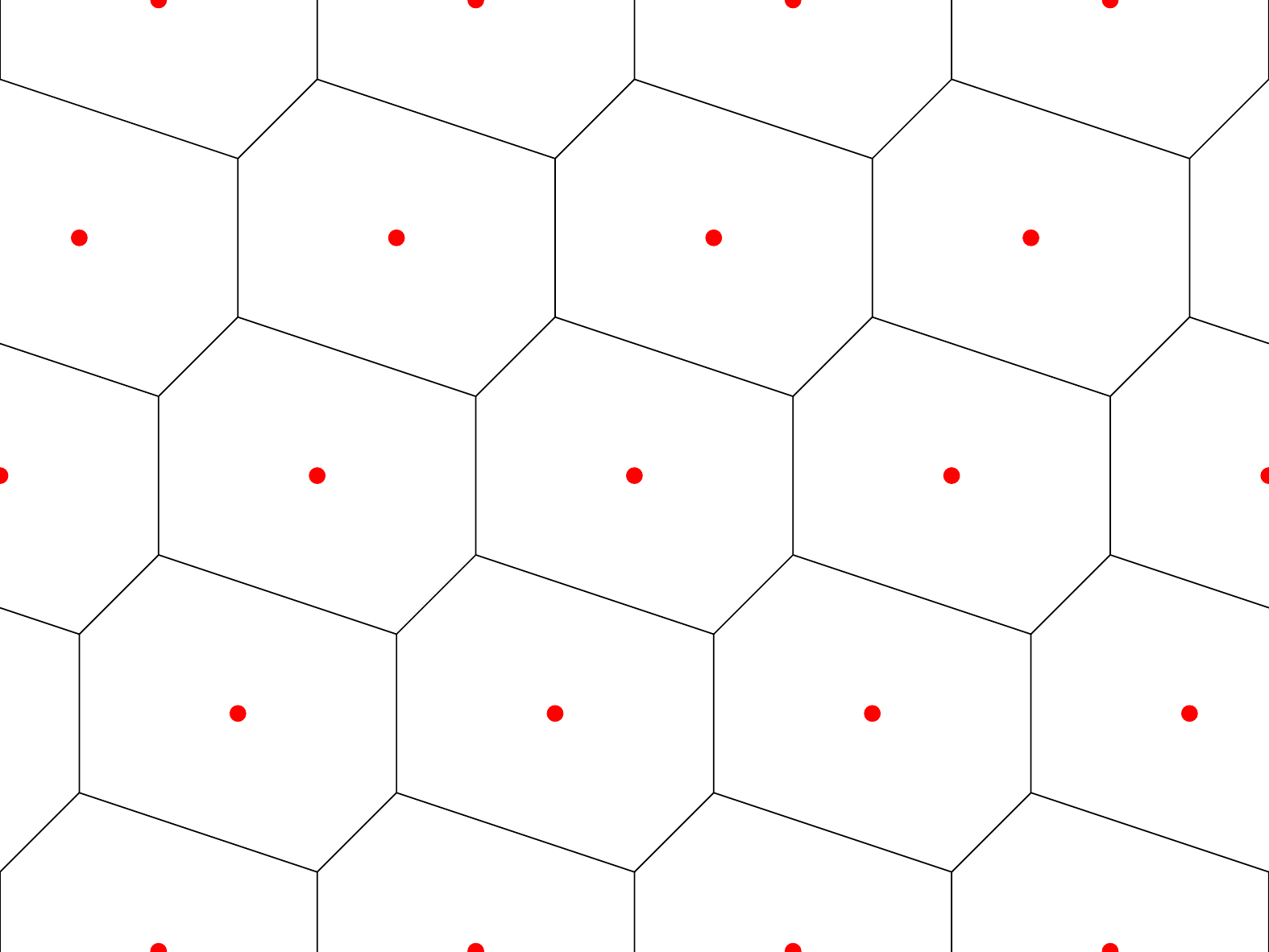}
\caption{The two kinds of Vorono\"\i\  regions of lattices in the plane.\label{Voro2}}
\end{figure}

We have already seen that $m_1(\R^2,\Vert \cdot \Vert_\infty )= \frac{1}{4}$, so 
it remains to deal with hexagons. Even though it is not true that every hexagonal Vorono\"\i\  region is linearly equivalent to the regular hexagon, we will first consider the regular hexagon in order to present in this basic case, the ideas that will be used in the general case.

\subsection{The regular hexagon}\label{HexaRegDemi}

The following result is due to Dmitry Shiryaev \cite{Dmitry}:

\begin{theo}\label{HexaRegTheo}
If $\mathcal{P}$ is the regular hexagon in the plane, then
$$m_1(\R^2,\Vert   \cdot\Vert  _\mathcal{P})=\frac{1}{4}.$$
\end{theo}

Let $\mathcal{P}$ be the regular hexagon in $\R^2$. We denote by $S$ its set of vertices and by $\partial\mathcal{P}$ its boundary. Thus, $\Vert  x\Vert  _\mathcal{P}=1$ if and only if $x\in \partial\mathcal{P}$.
We label the vertices of $\mathcal{P}$ modulo 6 as described in \fref{HexaRegGraph}.

The set $\frac{1}{2}S$ spans a lattice $V$. Let us consider $G_\mathcal{P}$, the subgraph of $G(\R^2, \Vert \cdot \Vert_\mathcal{P})$ induced by $V$. We shall prove that $\bar{\alpha}(G_\mathcal{P})\leq 1/4$. To do so, we introduce an auxiliary graph $\tilde{G}=(\tilde{V},\tilde{E})$, which is the Cayley graph with the same set of vertices $\tilde{V}=V$ corresponding to the generating set $\frac{1}{2}S$. In other words, for $x,y\in V$, $(x,y)\in \tilde{E}$  if and only if $x-y\in \frac{1}{2}S$. This graph is drawn in \fref{HexaRegGraph}.

We denote by $\tilde{d}(x,y)$ the distance between two vertices $x$ and $y$ in the graph $\tilde{G}$, \textit{i.e.} the minimal length of a path in $\tilde{G}$ between $x$ and $y$. We define the distance $\tilde{d}(A,B)$ in $\tilde{G}$ between two subsets of vertices $A$ and $B$ as the minimal distance between a vertex of $A$ and a vertex of $B$. The following lemma will be crucial for the proof of \tref{HexaRegTheo}:

\begin{figure}[!ht]
\includegraphics[scale=0.5]{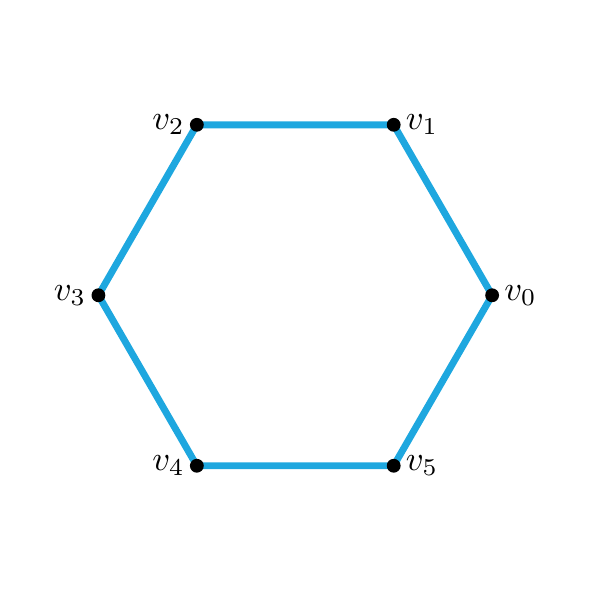}
\hspace{1cm}
\includegraphics[scale=0.5]{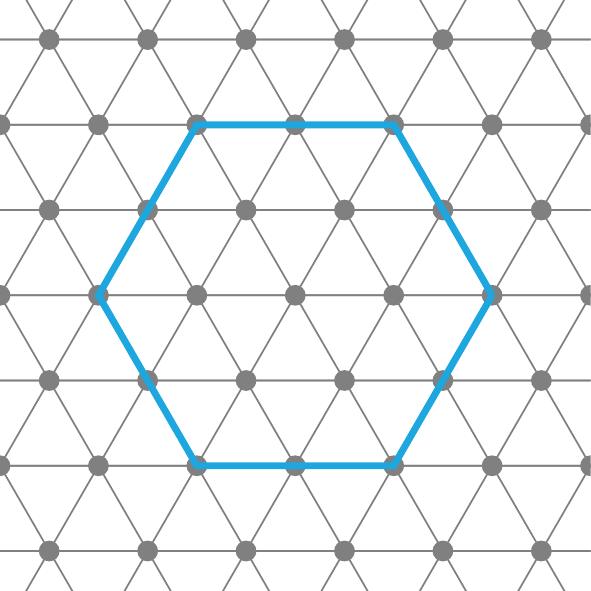}
\caption{The regular hexagon and the Cayley graph $\tilde{G}$. \label{HexaRegGraph}}
\end{figure}

\begin{lemm}\label{DP1DG2}
Let $u_1$ and $u_2$ be two vertices of $\tilde{G}$. Then: 
\begin{equation}\label{Property D}
\tilde{d}(u_1,u_2)=2 \Rightarrow \Vert  u_1-u_2\Vert  _\mathcal{P}=1. \tag{Property D} 
\end{equation}
\end{lemm}

\begin{proof}
Since $\tilde{G}$ is vertex-transitive, we may assume without loss of generality that $u_1=0$. The vertices $u$ at graph distance $2$ from $0$ must be of the form $\frac{v_i}{2}+\frac{v_j}{2}$. It is not hard to check that if $u_2=\frac{v_i+v_j}{2}$ is neither $0$ nor another $\frac{v_k}{2}$ (in which case $\tilde{d}(0,u_2)<2$), then it is a point of $\partial \mathcal{P}$ (see also \fref{HexaRegGraph}). 
\end{proof}

\begin{rema}
It can be noted, although it will not be useful here, that the equivalence $\tilde{d}(u_1,u_2)=2 \Leftrightarrow \Vert  u_1-u_2\Vert  _\mathcal{P}=1$ holds.
\end{rema}

For a set $A\subset \tilde{V}$, we define its closed neighborhood $$N[A]=\{ v\in \tilde{V} \text{ such that } \tilde{d}(v,A)\leq 1 \}=A+\left(\{0\}\cup \frac{1}{2}S\right).$$

Now we consider the \textit{cliques} of $\tilde{G}$, that is the sets $C\subset \tilde{V}$ such that for every $u\neq v\in C$, $\tilde{d}(u,v)=1$. We will use the following lemma several times: it shows that for any graph $\tilde{G}$ satisfying \eqref{Property D}, if $A\subset \tilde{V}$ avoids polytope distance 1, then $A$ is a union of cliques whose closed neighborhoods are disjoint:

\begin{lemm}\label{LemmeVoisinagesDisjoints} 
Let $\Vert  \cdot \Vert_\mathcal{P}$ be a polytope norm in $\R^n$, and $G_\mathcal{P}$ an induced subgraph of $G(\R^n,\Vert \cdot \Vert_\mathcal{P})$. Assume there exists an auxiliary graph $\tilde{G}$ with the same vertices $V$ as $G_\mathcal{P}$ satisfying \eqref{Property D}. Let $A\subset V$ avoiding polytope distance $1$. Then $A$ may be written as a union of cliques of $\tilde{G}$
$$ A=\bigcup_{C\in\mathcal{C}} C $$ 
such  that  if $C,C'\in \mathcal{C}$ with $C\neq C'$, then 
$$N[C]\cap N[C']=\emptyset.$$
\end{lemm}

\begin{proof}
Let us consider the decomposition of $A$ in connected components with respect to $\tilde{G}$. Following \lref{DP1DG2}, since $A$ avoids polytope distance 1, a connected component $C$ cannot contain two vertices at graph distance 2 from each other. So $C$ must be a clique. 

Assume that two different cliques $C$ and $C'$ of $A$ share a common neighbor. Thus $\tilde{d}(C,C')\leq 2$. Since $C$ and $C'$ are two disjoint connected components, $\tilde{d}(C,C')>1$. So $\tilde{d}(C,C')=2$, which is impossible, since $A$ avoids polytope distance 1. 
\end{proof}

Now we define the local density of a clique $C$ of $\tilde{G}$: $\delta^0(C)=\frac{|C|}{|N[C]|}$. In the next lemma, we analyse the different possible cliques of the graph $\tilde{G}$ that we constructed for the regular hexagon, and determine their local density:
\begin{lemm}\label{HexaRegDelta*}
For every clique $C\subset \tilde{G}$, 
$$ \delta^0(C)\leq \frac{1}{4}$$
\end{lemm}
\begin{proof}

Let $C$ be a clique of $\tilde{G}$. Since $\tilde{G}$ is vertex transitive, we can assume without loss of generality that $0\in C$. Up to the action of the dihedral group $\mathcal{D}_3$ on $V$, there are only three possible cliques in $\tilde{G}$ containing $0$, and one can easily determine their neighborhoods (see \fref{HexaRegCliques}):

\begin{itemize}
\item $C=\{0\}$: its neighborhood is $\{0\}\cup \frac{1}{2} S$. Thus $\delta^0(C)=\frac{1}{7}$.

\item $C=\left\{0,\frac{v_0}{2}\right\}$, and $\delta^0(C)=\frac{2}{10}=\frac{1}{5}$.

\item $C=\left\{0,\frac{v_0}{2},\frac{v_1}{2}\right\}$,
and $\delta^0(C)=\frac{3}{12}=\frac{1}{4}$.
\end{itemize}

\begin{figure}[!ht]
\includegraphics[scale=0.5]{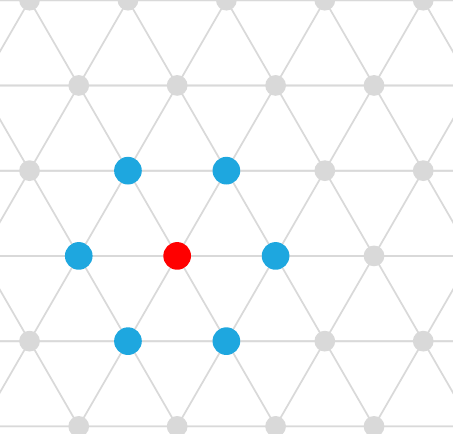}\hspace{.6cm}
\includegraphics[scale=0.5]{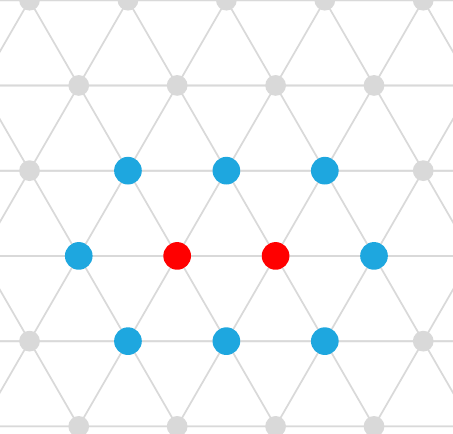}\hspace{.6cm}
\includegraphics[scale=0.5]{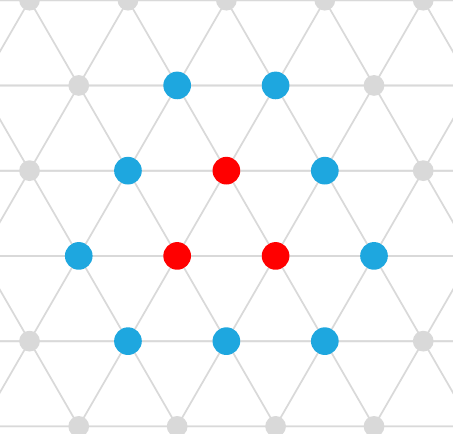}
\caption{The possible cliques and their neighborhood.\label{HexaRegCliques}}
\end{figure}

\end{proof}

We have all the ingredients to prove that the density of a set avoiding $1$ for the regular hexagon can not exceed $1/4$:

\begin{proof}[Proof of \tref{HexaRegTheo}] 
Following \lref{LemmeSousGraphe}, it is sufficient to prove $\bar{\alpha}(G_\mathcal{P})\leq\frac{1}{4}$.
If $A\subset V$ is a set avoiding 1, it may be written as the union of cliques in $\tilde{G}$, whose neighborhoods are disjoint (\lref{LemmeVoisinagesDisjoints}). So the density of $A$ is upper bounded by the maximum local density of a clique in $\tilde{G}$. So, from \lref{HexaRegDelta*}, $\bar{\alpha}(G_\mathcal{P})\leq\frac{1}{4}$. 
\end{proof}

\subsection{General hexagonal Vorono\"\i\  regions}\label{SoussecHexaIrreg}

In this subsection, we deal with a general hexagonal Vorono\"\i\  region $\mathcal P$ of the plane, and prove:

\begin{theo}\label{HexaIrregTheo}
If $\mathcal{P}$ is an hexagonal Vorono\"\i\  region in the plane, then
$$m_1(\R^2,\Vert   \cdot\Vert  _\mathcal{P})=\frac{1}{4}.$$
\end{theo}

Let $\mathcal{P}$ be the hexagonal Vorono\"\i\  region of a lattice $L\subset \R^2$. 
Let $\{\beta_0,\beta_1\}$ be a basis of $L$ such that the vectors $\beta_0$ , $\beta_1$, $\beta_2=\beta_1 - \beta_0$, and their opposites define the faces of $\mathcal{P}$.
We label the vertices $v_i$, for $0\leq i\leq 5$, of $\mathcal{P}$ in such a way that $\beta_i=v_i + v_{i+1}$, where $i$ is defined modulo $6$. This situation is depicted in \fref{SituationHexaIrreg}.

\begin{figure}[!ht]
\includegraphics[scale=0.7]{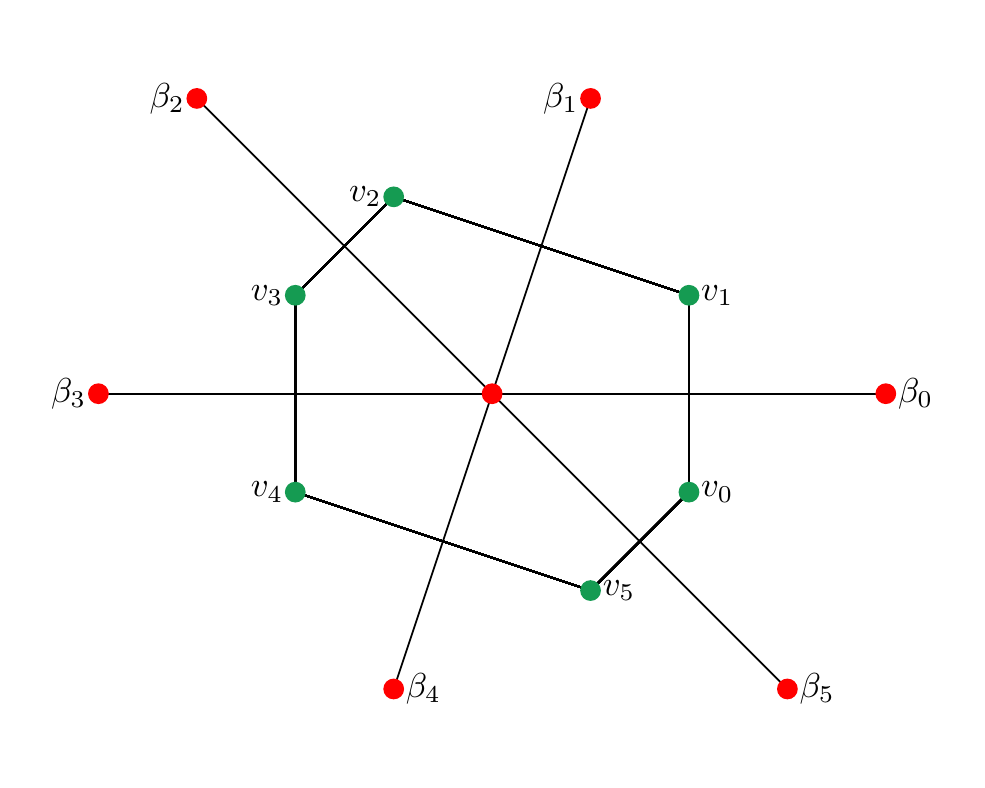}
\caption{The vectors $\beta_i$ and the vertices of the hexagon.\label{SituationHexaIrreg}}
\end{figure}

In order to prove \tref{HexaIrregTheo}, just like in the case of the regular hexagon, we shall construct a graph $G_\mathcal{P}$ induced by $G(\R^2, \Vert\cdot\Vert_\mathcal{P})$, and prove that $\bar{\alpha}(G_\mathcal{P})\leq 1/4$. 
Unfortunately, in  general, the vertices of $\mathcal{P}$ do not span a lattice. We will use a different point of view in order to build $G_\mathcal{P}$, together with an auxiliary graph $\tilde{G}$ that will satisfy a weaker version of \eqref{Property D}.

For the set $V$ of vertices of $G_{\mathcal P}$, we take the lattice  $\frac{1}{2}L$, 
together with the  translates of  the vertices $V_{\mathcal P}$ of $\mathcal{P}$ by  $\frac{1}{2}L$. 
We set $A=\frac{1}{2}L$ and $B=V_{\mathcal P}+\frac{1}{2}L$ so that $V=A\cup B$;  this construction is represented in \fref{HexaIrregSommets}
where the vertices of $A$ are depicted in red, and those of $B$ in green.

\begin{figure}[!ht]
\includegraphics[scale=0.3]{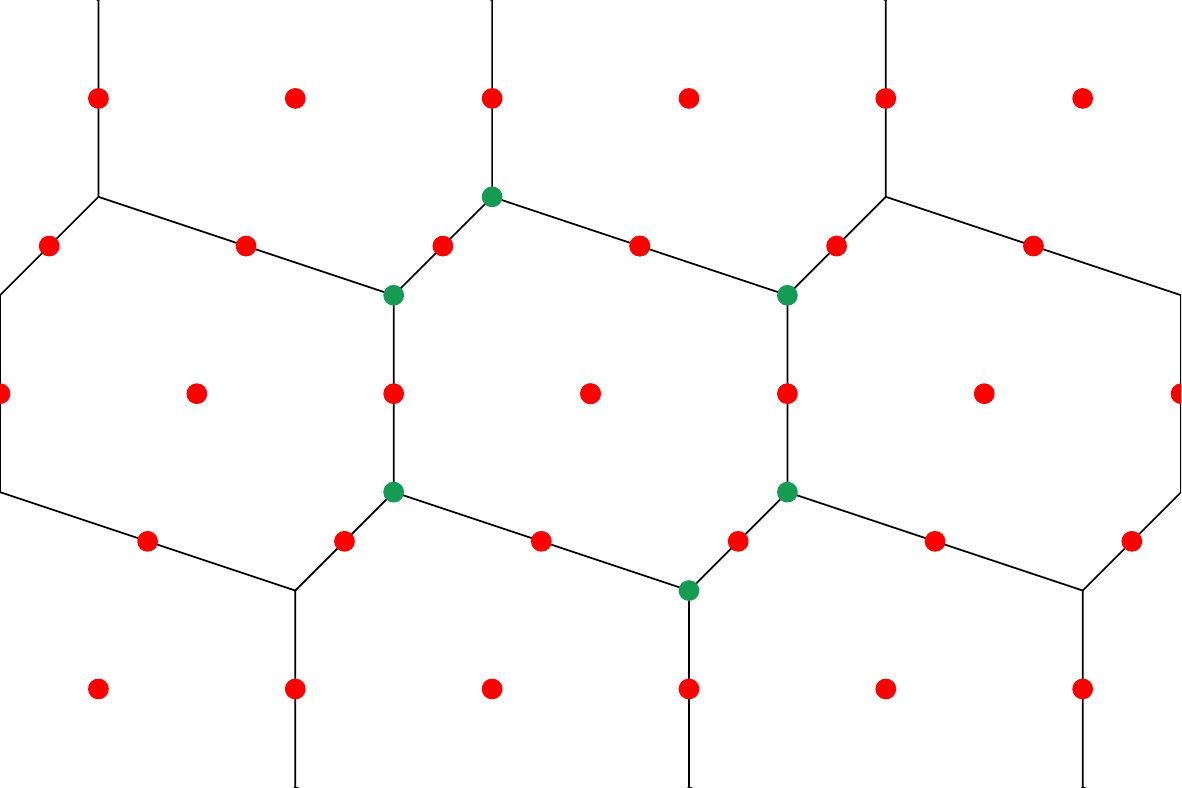} 
\hspace{.3cm}
\includegraphics[scale=0.3]{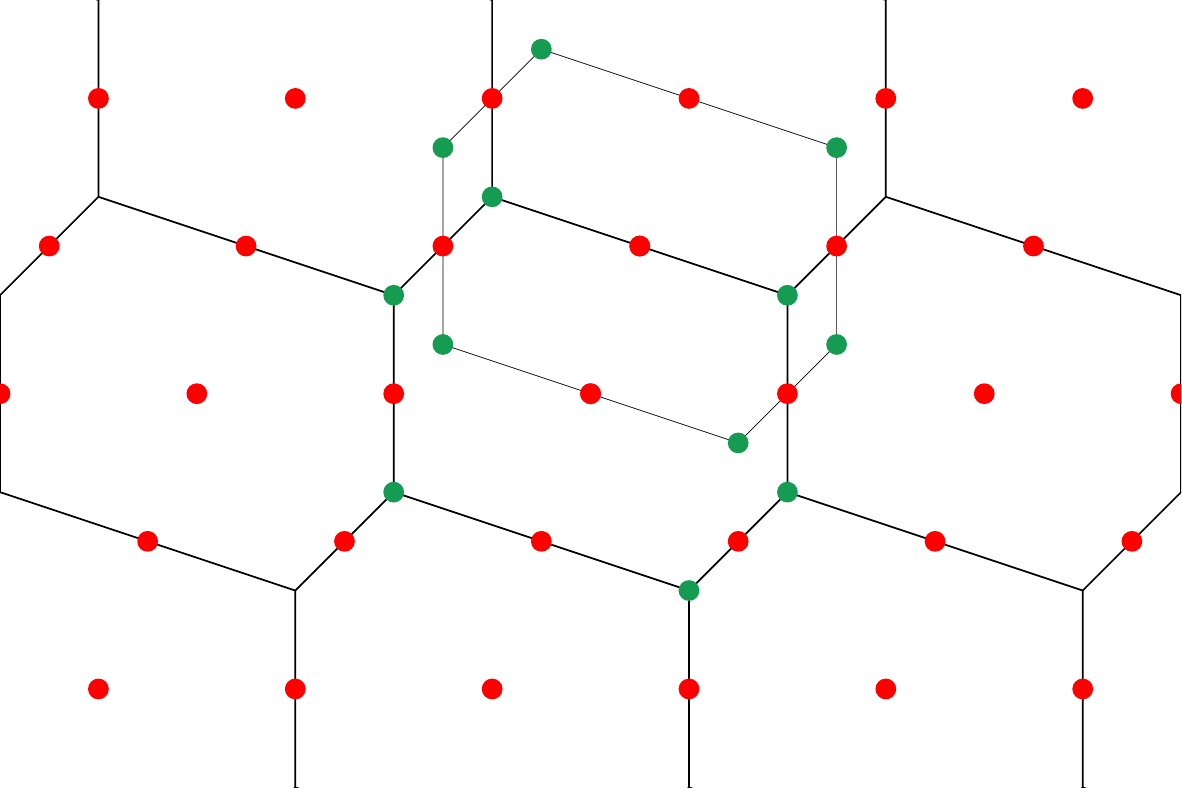} 
\hspace{.3cm}
\includegraphics[scale=0.3]{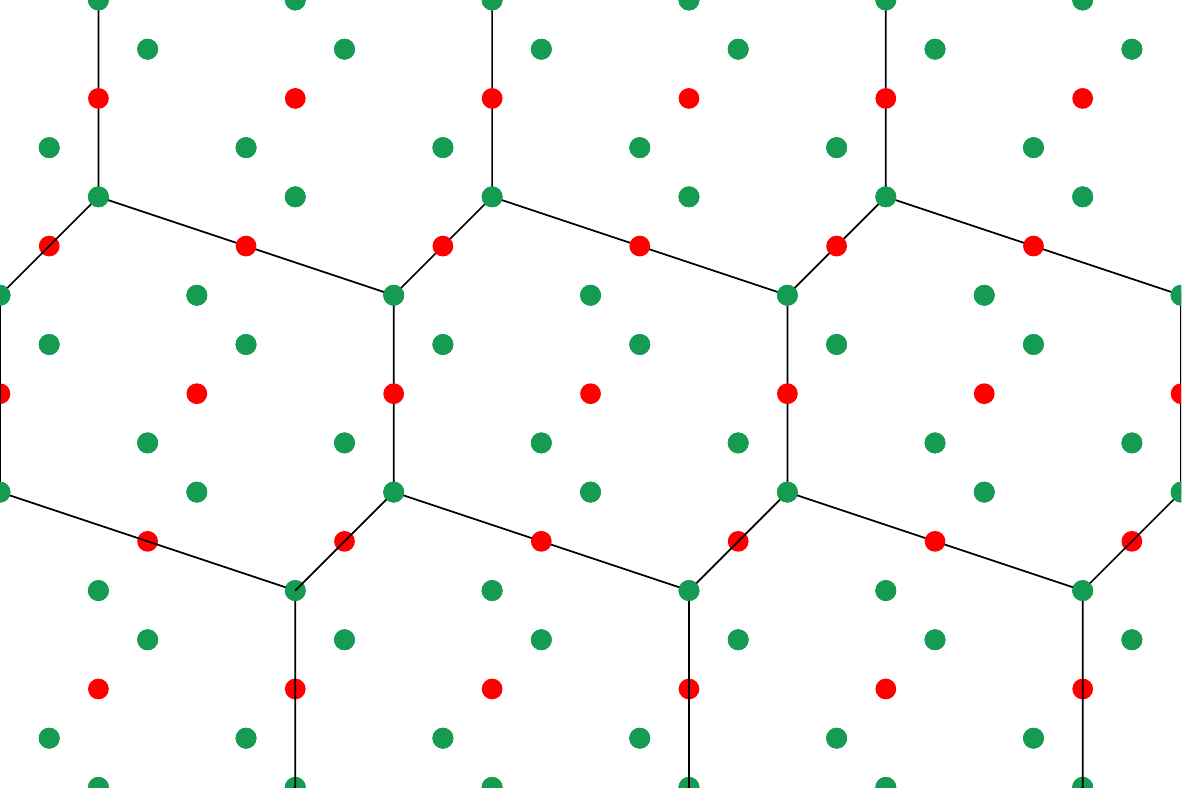}
\caption{Constructing the set of vertices of $G_\mathcal{P}$}\label{HexaIrregSommets}
\end{figure}

Let us note that for every $i$, $v_{i+2}= v_i \mod L $. Indeed, $$v_{i+2}-v_i=v_{i+2}+v_{i+1}-(v_i+v_{i+1})=\beta_{i+1} - \beta_i = \beta_{i+2}.$$ 

As a consequence, we may write $V$ as the disjoint union of three sets:
$$V= \frac{1}{2}L \cup (\frac{1}{2}L + v_0) \cup (\frac{1}{2}L + v_1),$$
and this implies that the density of $B$ in $V$ is twice that of $A$.

Now, let us construct the auxiliary graph $\tilde{G}=(\tilde{V},\tilde{E})$. It has the same vertices as $G_{\mathcal P}$, i.e. $\tilde{V}=V$. Let us describe the edges of $\tilde{G}$. By construction, there are exactly $7$ vertices of $V$ in the interior of $\mathcal{P}$: the center $0\in A$, and six points of $B$ denoted $s_0,\ldots,s_5$, with
$$s_i=\frac{v_{i-1}+v_{i+1}}{2}.$$
For every point of $a\in A$, we define the edges $(a,a+s_i)$ and $(a+s_i, a+s_{i+1})$ for $i$ from $0$ to $5$. This is illustrated in \fref{HexaIrregAretes}.

\begin{figure}[!ht]
\includegraphics[scale=0.7]{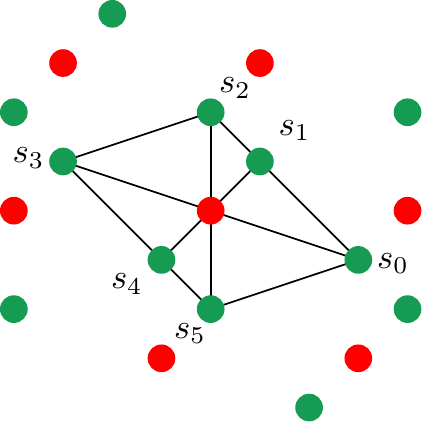} \hspace{1.5cm}
\includegraphics[scale=0.45]{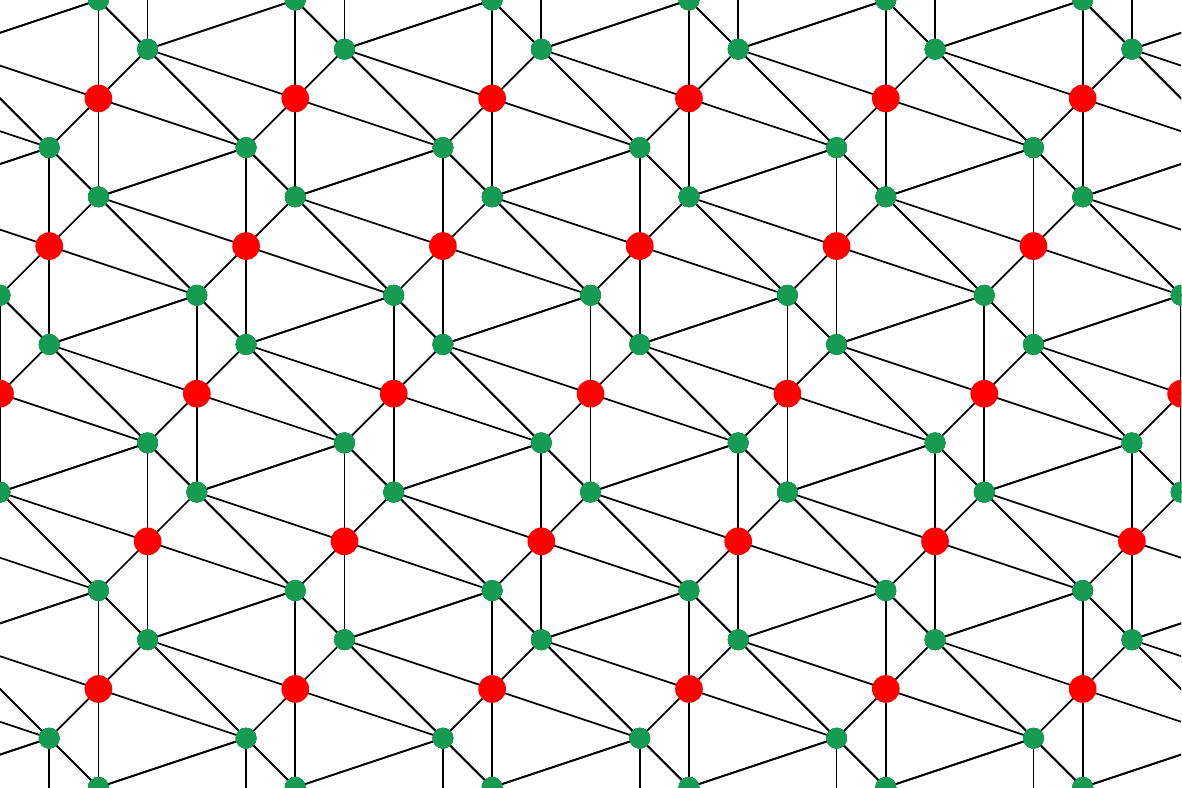} 
\caption{Constructing  the edges of $\tilde{G}$.\label{HexaIrregAretes}}
\end{figure}

\begin{rema}
In the case of the regular hexagon, this construction leads to the same graph $\tilde{G}$ that we considered in Subsection \ref{HexaRegDemi}.
\end{rema}

Let us describe the neighborhood (with respect to $\tilde{G}$) of each type of point. By construction, a point in $A$ has $6$ neighbors, and they all belong to $B$. 
A vertex $a+s_i$ of $B$ also has six neighbors. Three of them are elements of $A$, namely $a$, $a+\frac{\beta_i} 2$ and $a+\frac{\beta_{i-1}} 2$ and the other three are elements of $B$, namely, $a+s_{i-1}$, $a+s_{i+1}$ and $a+v_i$. \fref{aretesV} illustrates the neighborhoods of the vertices of $\tilde{G}$.

\begin{figure}[!ht]
\includegraphics[scale=0.9]{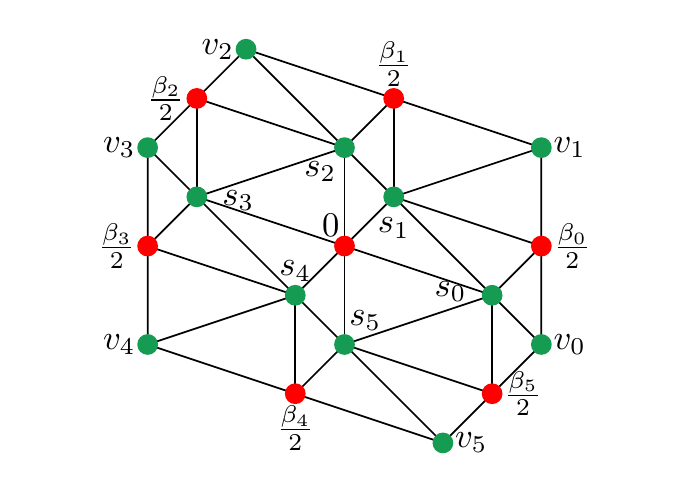} \hspace{10pt}
\caption{The basic pattern in $\tilde{G}$.	\label{aretesV}}
\end{figure}

It should be noted that \eqref{Property D} is not in general fulfilled by $\tilde{G}$: indeed, the vertices $s_0$ and $s_3$ are at graph distance $2$ in $\tilde{G}$  but not (in general) at polytope distance $1$. However, this property continues to hold \emph{for points that share a common neighbor in $B$}. We prove this in the next  lemma, which  will play the role of \lref{DP1DG2} for this new graph $\tilde{G}$:

\begin{lemm}\label{DP1DG2HexaIrreg}
If two vertices $x,y\in V$ are at distance 2 from each other in $\tilde{G}$ and have a common neighbor $z\in B$, then $\Vert  x-y\Vert  _\mathcal{P}=1$.
\end{lemm}

\begin{proof}
First suppose that at least one of the two vertices is in $A$. In this case we may assume $x=0$. Then $z$ is one of the $s_i$, and following the analysis of the neighbors of $s_i$, $y$ must be in the set $\{ 0, s_{i-1}, s_{i+1}, \frac{ \beta_{i} }{2}, \frac{ \beta_{i-1} }{2}, v_i\}$. The first three are obviously not at graph distance $2$ from $0$, so $y$ is one of the last three vertices, and they all are in $\partial \mathcal{P}$. Thus, $\Vert  x-y\Vert  _\mathcal{P}=1$.

Now suppose $x,y,z\in B$. Then we may assume without loss of generality $x=s_{i-1}$, and $z=s_{i}$. Since $z$ has only three neighbors in B, $y$ can be either $s_{i+1}$ or $v_i$. We have:
$$s_{i+1} -s_{i-1} =\frac{v_{i}+v_{i+2}}{2} - \frac{v_{i}+v_{i-2}}{2} = \frac{v_{i+2}-v_{i-2}}{2} = \frac{v_{i+2}+v_{i+1}}{2} = \frac{\beta_{i+1}}{2}  $$ 
and
$$v_{i} -s_{i-1} =v_{i} - \frac{v_{i}+v_{i-2}}{2} = \frac{v_{i}-v_{i-2}}{2} = \frac{v_{i}+v_{i+1}}{2} = \frac{\beta_{i}}{2}. $$
In both cases $\Vert  x-y\Vert  _\mathcal{P}=1$. 
\end{proof}

Let $U\subset V$ be a set of vertices avoiding polytope distance 1, let $C$ be a connected component of $U$ and let $N[C]$ be its closed neighborhood. We define:

$$ N_B[C] = N[C]\cap B $$
and 
$${\delta}^0_B(C)=\frac{|C|}{|N_B[C]|}.$$

The following lemma is the analogue of \lref{LemmeVoisinagesDisjoints} in this situation: we show that if $C$ and $C'$ are two different connected components, then $N_B[C]$ and $N_B[C']$ must be disjoint:
\begin{lemm}\label{HexaIrregVoisDisjoints}
Let $U\subset V$ be a set avoiding polytope distance $1$. If $C\neq C'$ are two connected  components of $U$, then
$$ N_B[C]\cap N_B[C'] =\emptyset. $$
\end{lemm}

\begin{proof}
If a vertex $z\in B$ is in both $N_B[C]$ and $N_B[C']$, then there is $x\in C$, $y\in C'$ such that $\tilde{d}(x,z)=\tilde{d}(z,y)=1$. Since $C$ and $C'$ are connected components of $U$, we have $\tilde{d}(x,y)>1$. Thus $\tilde{d}(x,y)=2$ and by  \lref{DP1DG2HexaIrreg}, $\Vert  x-y\Vert  _\mathcal{P}=1$, which is impossible, since $U$ avoids $1$. 
\end{proof}

Now we study the different possible connected components:

\begin{lemm}\label{HexaIrregDelta}
Let $U\subset V$ be a set avoiding polytope distance $1$. If $C$ is a connected component of $U$, then
$$ \delta^0_B\leq \frac{3}{8}. $$
\end{lemm}

\begin{proof}
We enumerate the possible connected components. Let us start with the isolated points. Up to translations by $\frac{1}{2}L$, we have: 
\begin{itemize}
\item $C=\{0\}\subset A$.  Its neighborhood is made of six vectors from $B$. So $\delta_B^0(C)=1/6$.
\item $C=\{ s_i\}\subset B$.  We know that such a vertex has three neighbors in $B$, thus $\delta_B^0(C)=1/4$.
\end{itemize}

\begin{figure}[!ht]
\includegraphics[scale=0.6]{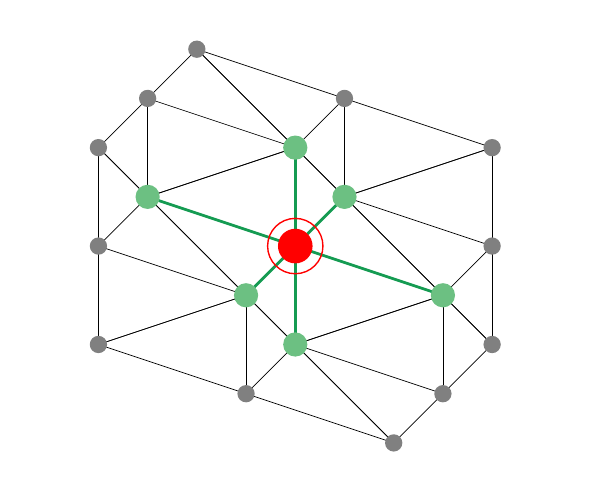} \hspace{10pt}
\includegraphics[scale=0.6]{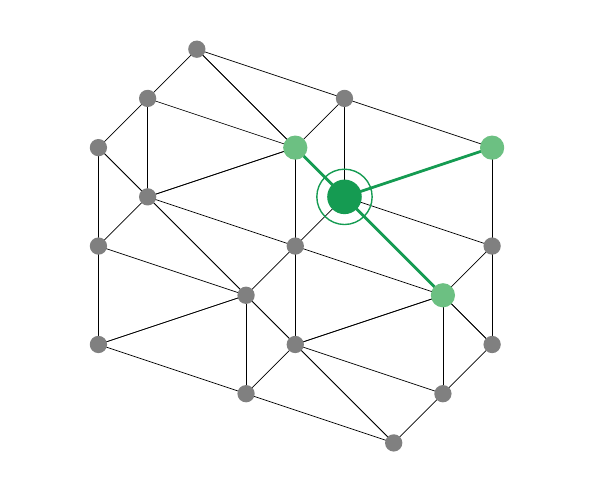} 
\caption{The two possible types of connected component with one element. The circled vertices denote the elements of $C$ and the figure represents all their neighbors in $B$.\label{HexaIrregCC1}}
\end{figure}

We now focus on the connected components of size 2. Since a vertex in $A$ has all its neighbors in $B$, such a connected component can not contain two elements of $A$. Thus, up to translation, we only have: 
\begin{itemize}
\item  $C=\{0,s_i\}$, and the only neighbor in $B$ that is not a neighbor of $0$ is $v_i$. Thus $\delta^0_B=2/7$. 
\item $C=\{ s_i, s_{i+1} \}$ and the neighbors in $B$ are $s_{i-1},v_i,s_{i+2},v_{i+1}$. Thus $\delta^0_B=2/6=1/3$.
\end{itemize}

\begin{figure}[!ht]
\includegraphics[scale=0.6]{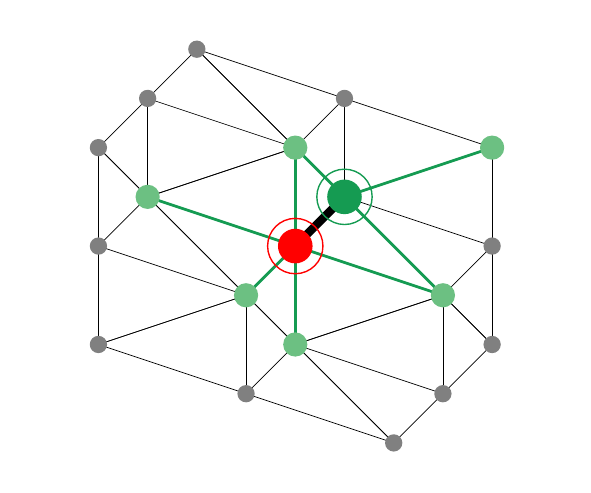} \hspace{10pt}
\includegraphics[scale=0.6]{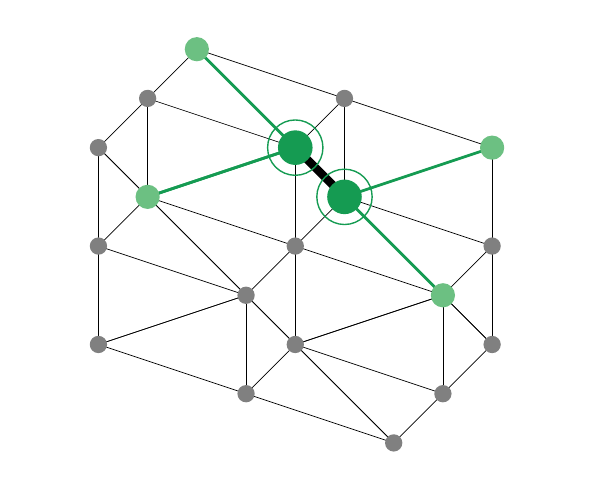} 
\caption{The two possible types of connected component with two elements.\label{HexaIrregCC2}}
\end{figure}

There are up to translations two kinds of connected components of size three:

\begin{itemize}
\item $C= \{0,s_i,s_{i+1}\}$. The only neighbor of $s_{i+1}$ in $B$ that is not a neighbor of $\{0,s_i\}$ is $v_{i+1}$. Thus $\delta^0_B=3/8$.
\item $C= \{0,s_i,-s_{i}\}$. The only neighbor of $-s_{i}$ in $B$ that is not a neighbor of $\{0,s_i\}$ is $-v_{i}$. Thus $\delta^0_B=3/8$. 
\end{itemize}

\begin{figure}[!ht]
\includegraphics[scale=0.6]{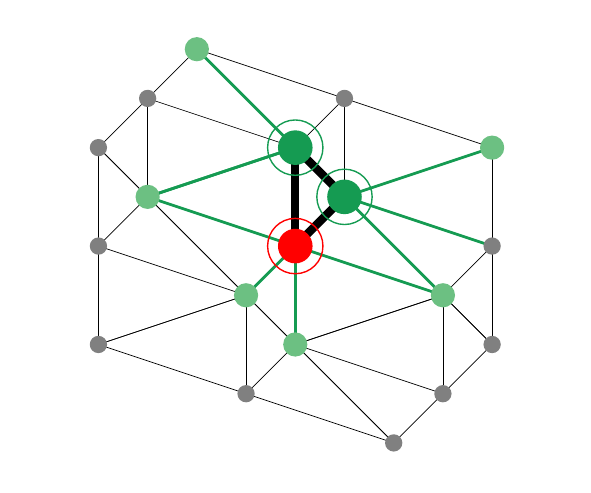} \hspace{10pt}
\includegraphics[scale=0.6]{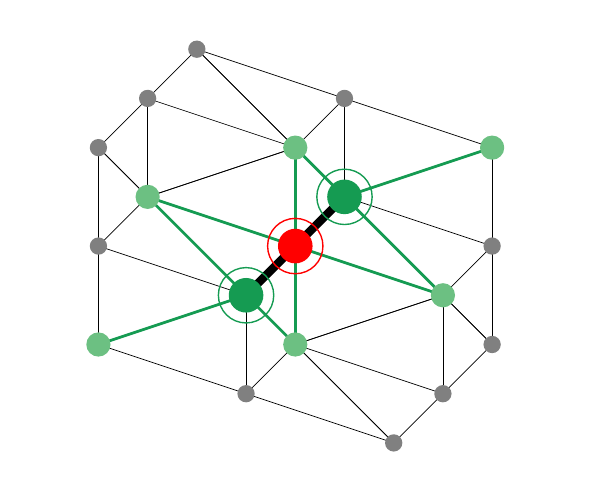} 
\caption{The two possible types of connected component with three elements.\label{HexaIrregCC3}}
\end{figure}

It is easy to check, applying \lref{DP1DG2HexaIrreg}, that we have enumerated all kind of connected components
of $U$. 
\end{proof}

Finally we can put everything together and complete the proof of \tref{HexaIrregTheo}:

\begin{proof}[Proof of \tref{HexaIrregTheo}] 
Let $U\subset V$ avoiding polytope distance 1.
We define
$$ \delta_B(U)=\limsup_{R\to\infty} \frac{|U\cap V_R|}{|B\cap V_R|} $$
where as usual $V_R=V\cap [-R,R]^n$. We have:
$$ \delta_{G_{\mathcal P}}(U)=\delta_B(U) \times \delta_{G_{\mathcal P}} (B), $$
and since $V=A\cup B$ and $B$ is twice as dense as $A$ in $G_{\mathcal P}$, 
$$ \delta_{G_{\mathcal P}}(U)=\frac{2}{3}\delta_B(U). $$
From \lref{HexaIrregVoisDisjoints}, we have $\delta_B(U)\leq \sup_{C\subset U} \delta_B^0(C)$ where $C$ runs over the connected components of $U$. Then \lref{HexaIrregDelta} shows that
$$ \delta_B(U)\leq \frac{3}{8} $$ 
and  we get
$$ \delta_{G_{\mathcal P}}(U) \leq \frac{2}{3} \times \frac{3}{8} =\frac{1}{4}. $$
\end{proof}

\section{The norms associated with the Vorono\"\i\  regions of the lattices $A_n$ and $D_n$}\label{SectionFamilles}

\subsection{The lattice $A_n$} \label{SoussecAn}

Here we consider for any $n\geq 2$, the lattice $$A_n=\Z^{n+1}\cap H,$$ where $H$ is the hyperplane $H=\{(x_1,\ldots,x_{n+1})\in \R^{n+1}\mid \sum_{i=1}^{n+1} x_i=0\}$. Let $\mathcal{P}$ be the Vorono\"\i\  region of $A_n$. We shall prove:

\begin{theo}\label{AnTheo}
For every dimension $n\geq 2$, if $\mathcal{P}$ is the Vorono\"\i\  region of the lattice $A_n$, then
$$m_1(\R^n,\Vert   \cdot\Vert  _\mathcal{P})=\frac{1}{2^n}.$$
\end{theo}

 In fact, for $n=2$, the Vorono\"\i\  region of $A_2$ is nothing but the regular hexagon. We are going to generalize to all dimensions $n\geq 2$ the strategy that we used in subsection \ref{HexaRegDemi}.

Let us recall the description of the Vorono\"\i\  region $\mathcal{P}$
of $A_n$ given in \cite[Chapter 21, section 3]{Conway:1987:SLG:39091}. 

The orthogonal projection on $H$ is denoted by $p_H$.  Let, for $1\leq i\leq n$ and $j:=(n+1)-i$, 
$$\begin{aligned} v_i&=
 p_H((\underbrace{0,\ldots,0}_{i \text{ times}},\underbrace{1,\ldots ,1}_{j \text{ times}}))\\
&= ({0,\ldots,0},{1,\ldots ,1}) - \frac{j}{n+1} (1,\ldots ,1) \\ 
&=(\underbrace{\frac{-j}{n+1},\ldots,\frac{-j}{n+1}}_{i \text{ times}},\underbrace{\frac{i}{n+1},\ldots, \frac{i}{n+1}}_{j \text{ times}}).\\ 
 \end{aligned}$$

Let $S$ be the simplex whose vertices are $0$ and the vectors $v_i$. 
Then the vertices of $\mathcal{P}$ are the images of the non zero vertices of $S$ under the permutation group $\mathfrak{S}_{n+1}$. In other words, the set of vertices of $\mathcal{P}$ is 
$$V_\mathcal{P}=\{p_H(u) \mid u\in V_0 \}, \text{ where } V_0=\{0,1\}^{n+1} \setminus \{ (0,\ldots,0),(1,\ldots,1) \}.$$

We also analyze the boundary of $\mathcal{P}$, in order to understand the norm associated with $\mathcal{P}$. The non zero vertices of $S$ are supported by the hyperplane $H_{0,n}$ of $H$ defined by $H_{0,n}=\{ x=(x_0,\ldots x_n)\in H\mid x_n-x_0=1\}$. Applying $\mathfrak{S}_{n+1}$, we find that the faces of $\mathcal{P}$ are supported by all the $H_{i,j}=\{ x=(x_0,\ldots x_n)\in H\mid x_j-x_i=1\}$, for $i\neq j$. So $$\begin{cases} & x\in \mathcal{P} \text{ if and only if for all } i\neq j,\  x_j-x_i\leq 1 \\ &  x\in \partial \mathcal{P} \text{ if and only if  } \max_{i\neq j} (x_j-x_i) = \max_j x_j - \min_i x_i =1, \end{cases}$$
and more generally the norm $\Vert x\Vert_\mathcal{P}$ of a vector $x\in H$ is given by
$$ \Vert  x\Vert  _\mathcal{P} = \max_j x_j - \min_i x_i .$$
Note that if $x=p_H(u)$, because $H^{\perp}=\R (1,\dots,1)$, we have  $${\max_j x_j - \min_i x_i = \max_j u_j - \min_i u_i}.$$

The vertices of $\mathcal{P}$ generate a lattice, which is the dual lattice of $A_n$:

\begin{lemm}
The vertices of $\mathcal{P} $ span over $\Z$ the lattice  $ A_n^\# = p_H(\Z^{n+1}) $.
\end{lemm}

\begin{proof}
Let $v\in V_{\mathcal{P}}$. There is $u\in V_0$ such that $v=p_H(u)$. Since $(u-p_H(u))\in H^\perp$, we have, for every $x\in A_n = \Z^{n+1}\cap H$,
$$ \begin{aligned} \langle x, v\rangle = \langle x, p_H(u) \rangle = \langle x, u\rangle = \sum_{i,u_i=1} x_i \in \Z ,\end{aligned}$$
so $\Span_\Z(V_\mathcal{P})\subset A_n ^\#$.

Now let us take $x\in A_n^\# = p_H(\Z^{n+1})$, so $x=p_H(z_0,\ldots,z_n)$, with $z_i\in \Z$. Then 
$$x = p_H(z_0,\ldots,z_n)=\sum_{i=0}^n z_ip_H(0,\ldots,0,\underbrace{1}_i,0\ldots, 0)\in \Span_\Z(V_\mathcal{P}).$$

Thus $\Span_\Z(V_\mathcal{P})=A_n ^\#$.
\end{proof}

We consider  the subgraph $G_\mathcal{P}$ of $G(\R^n,\Vert \cdot \Vert_\mathcal{P})$ induced by the set of vertices $\frac{1}{2} A_n ^\#$, and  the auxiliary graph $\tilde{G}$ which is the Cayley graph on $\frac{1}{2} A_n ^\#$
associated with the generating set $\frac{1}{2}V_\mathcal{P}$. These graphs  are the generalizations of the graphs that we considered in subsection \ref{HexaRegDemi}. Here we show that $\tilde{G}$ satisfies the same remarkable property:

\begin{lemm}\label{DP1DG2An}
The graph $\tilde{G}$ satisfies \eqref{Property D}.
\end{lemm} 

\begin{proof}
We follow the proof of \lref{DP1DG2}. We may assume $x=0$, and we need to show that, for $v,v'\in V_{\mathcal{P}}$, if $\frac{v+v'}{2}\neq 0$ then it is either some $\frac{v''}{2}\in \frac{1}{2}V_{\mathcal{P}}$, or an element of $\partial \mathcal{P}$. Equivalently, we study $v+v'$ and show that one of the three following situations occurs:
$$\begin{cases} & v+v'=0, \\ & v+v'=v''\in V_\mathcal{P}, \\& \Vert  v+v'\Vert  _\mathcal{P}=2. \end{cases}$$

Let $u$ and $u'$ be elements of $V_0=\{0,1\}^{n+1} \setminus \{ (0,\ldots,0),(1,\ldots,1)\}$ such that $v=p_H(u)$ and $v'=p_H(u')$. The coordinates of the vector $u+u'$ belong to $\{0,1,2\}$, but  cannot be all $0$ nor all $2$. We explore the  possible cases:
\begin{itemize}
\item If $u+u'=(1, \ldots, 1)$, then $p_H(u+u')=(0, \ldots ,0)$, and $v+v'=0$.
\item If the coordinates of $u+u'$ are only $0$'s and $1$'s, then $u+u'\in V_0$, and thus $v+v'\in V_\mathcal{P}$.
\item If the coordinates of $u+u'$ are only $1$'s and $2$'s, we may decompose $u+u'$ as ${u+u'= (1,\ldots,1) + w}$, and $w$ must be an element of $V_0$. This implies that $v+v'=p_H(w)\in V_\mathcal{P}$.
\item The last  remaining case is when both $0$'s and $2$'s appear in the coordinates of $U=u+u'$. 
Then, $\max_j U_j - \min_i U_i=2$, that is $\Vert  v+v'\Vert  _\mathcal{P}=2$.
\end{itemize}
\end{proof}

Because $\tilde{G}$ satisfies \eqref{Property D},  \lref{LemmeVoisinagesDisjoints} is satisfied by $\tilde{G}$. So we can proceed to analyze  the cliques of $\tilde{G}$, and for each of them, determine its local density. Since $\tilde{G}$ is vertex transitive, we only describe the cliques containing $0$. For $u\in V_0$, we define its support  $I=\{ i \in \{1,\ldots,n+1\}, \ u_i=1\}$.

\begin{lemm}
The cliques of $\tilde{G}$ containing $0$ are the sets of the form $$\left\{0,\frac{p_H(u_1)}{2},\ldots,\frac{p_H(u_s)}{2}\right\}$$ such that if $I_i$ is the support of $u_i$, then 
$$ I_1\subset I_2 \subset \ldots \subset I_s.$$
In particular, since $s\leq n$, a clique can not contain more than $n+1$ vertices. 
\end{lemm}

\begin{proof}
Let $C$ be a clique of $\tilde{G}$, and assume $0\in C$. Then the other elements of $C$ must belong to $\frac{1}{2}V_\mathcal{P}$ and since $C$ is a clique, they must be adjacent in the graph. In other words, if $\frac{v}{2},\frac{v'}{2}\in C$, then $\frac{v-v'}{2}\in \frac{1}{2}V_\mathcal{P}$. Let $v\neq v'\in V_\mathcal{P}$, and $u,u'\in V_0$ such that $v=p_H(u)$ and $v'=p_H(u')$. We denote by $I$ and $I'$ the respective supports of $u$ and $u'$. For $i\in\{1,\ldots,n+1\}$, the $i$th coordinate of $u-u'$ is:
$$\begin{cases} 
 1 & \text{ if } i\in I\setminus I', \\
 -1 & \text{ if } i\in I'\setminus I, \\
 0 & \text{ otherwise. } 
 \end{cases}$$

If both $1$ and $-1$ appear in the coordinates of $u-u'$, then $\Vert  v-v'\Vert  _\mathcal{P}=2$, and $v-v'\notin V_\mathcal{P}$.  By definition of $V_0$ and since $v\neq v'$, the coordinates of $u-u'$ must take two 
different values. Two cases remain: if $u-u'$ contains only $0$'s and $1$'s, $u-u'\in V_0$ and $v-v'\in V_\mathcal{P}$; and if it contains only $0$'s and $-1$'s, then we can write $u-u'=w-(1,\ldots,1)$, with $w\in V_0$, so that $v-v'\in V_\mathcal{P}$ as well.

To conclude, we find that $v-v'\in V_\mathcal{P}$ if and only if $I\subset I'$ or $I'\subset I$. 

\end{proof}

\begin{lemm}\label{AnDelta*}
For every  clique $C$ of $\tilde{G}$,
$$\delta^0(C)\leq\frac{1}{2^n}.$$
\end{lemm}

\begin{proof}

Let $\left\{0,\frac{p_H(u_1)}{2},\ldots,\frac{p_H(u_s)}{2}\right\}$ be a clique. By symmetry, we may assume that $$u_i=(\underbrace{1,\ldots, 1}_{w_i},0,\ldots,0), $$
where $w_i=|I_i|$. We want to count the vertices in $$N[C]=\frac{1}{2}\left( \{0,p_H(u_1),\ldots,p_H(u_s)\} + V_\mathcal{P} \right).$$   Since $0\in C$, the set $\left( \{0,p_H(u_1),\ldots,p_H(u_s)\} + V_\mathcal{P} \right)$ must contain all the images of $V_0\cup \{0\}$ by $p_H$: there are ${2^{n+1}-1}$ such vertices. We count,
 for each $i=1,\dots,s$, how many new neighbors are provided by $p_H(u_i)+V_{\mathcal P}$. We find that
\begin{itemize}
\item The vector $$u_1=(\underbrace{1,\ldots, 1}_{w_1},0,\ldots,0),$$ provides ${(2^{w_1}-1)(2^{n+1-w_1}-1)}$ {new neighbors}.
\item The vector $$u_2=(\underbrace{1,\ldots, 1}_{w_1},\underbrace{1,\ldots, 1}_{w_2-w_1},0,\ldots,0),$$
 provides ${2^{w_1}(2^{w_2-w_1}-1)(2^{n+1-w_2}-1)}$ {new neighbors}.
\item For any $2\leq i \leq s$, the vector $u_i$ will provide ${2^{w_{i-1}}(2^{w_i-w_{i-1}}-1)(2^{n+1-w_i}-1)}$ {new neighbors}.
\end{itemize}
By summing all the values, if we set  $w_0=0$, we get:
$$\begin{aligned} 
|N[C]|&=2^{n+1}-1 + \sum_{i=1}^{s} 2^{w_{i-1}}(2^{w_i-w_{i-1}}-1)(2^{n+1-w_i}-1) \\
 & = (s+1)2^{n+1} - (\sum_{i=1}^{s}2^{n+1-(w_i-w_{i-1})}+ 2^{w_s}).
  \end{aligned}$$
Since $w_s\leq n$ and for every $i$, $(w_i-w_{i-1})\geq 1$, we have 
$$2^{w_s} +\sum_{i=1}^{s}2^{n+1-(w_i-w_{i-1})}\leq (s+1)2^n,$$
and this implies
$$|N[C]|\geq (s+1)2^{n+1} - (s+1)2^n = (s+1)2^n.$$
Finally, the local density of $C$ satisfies:
$$ \delta^0(C) = \frac{|C|}{|N[C]|} = \frac{s+1}{|N[C]|} \leq \frac{1}{2^n} ,$$
and we may note that this bound is sharp if and only if $w_s=n$ and for every $i$, $w_i-w_{i-1} = 1$, that is when $C$ is a maximal clique of the form
$$ \{0, (1,0,\ldots,0), (1,1,0,\ldots,0), \ldots ,(1,\ldots,1,0,0), (1,\ldots,1,0)\}.  $$ 
\end{proof}

Now we can conclude the proof of \tref{AnTheo}:

\begin{proof}[Proof of \tref{AnTheo}] 
Following \lref{AnDelta*} and \lref{LemmeVoisinagesDisjoints}, $\bar{\alpha}(G_\mathcal{P})\leq\frac{1}{2^n}$, which leads to the theorem, following \lref{LemmeSousGraphe}.
\end{proof}

\subsection{The lattice $D_n$, $n\geq 4$}

We apply the same method as for $A_n$ to another classical family of lattices. For $n\geq 4$, the lattice $D_n$ is defined by
$$D_n=\{ x=(x_1,\ldots,x_n)\in\Z^n\mid  \sum_{i=0}^n x_i = 0 \mod 2\}.$$ 

The same construction provides again a graph that satisfies $\eqref{Property D}$. Unfortunately, the analysis of the neighborhoods of the cliques does not lead to the wanted $\frac{1}{2^n}$ upper bound. Nevertheless, we can prove: 

\begin{theo}\label{DnTheo}
For every dimension $n\geq 4$, if $\mathcal{P}$ is the Vorono\"\i\  region of the lattice $D_n$, then
$$m_1(\R^n,\Vert   \cdot\Vert  _\mathcal{P})\leq \frac{1}{(3/{4})2^n+n-1}.$$
\end{theo}

Let us describe the Vorono\"\i\  region of $D_n$. Again we refer to \cite{Conway:1987:SLG:39091} for further details. Let $S$ be the simplex whose vertices are $0$, $(0,\ldots,0,1)$, $\left(\frac{1}{2}, \ldots, \frac{1}{2}\right)$, and the vectors $\left(\underbrace{0,\ldots,0}_{i},\frac{1}{2}, \ldots, \frac{1}{2}\right)$, for $2\leq i\leq n-2$. Then  the Vorono\"\i\  region $\mathcal{P}$ of $D_n$ is the union of the images of $S$ by the group generated by all permutations of the coordinates and sign changes of evenly many coordinates. 
Note that some of the vertices of $S$ are not extreme points of $\mathcal{P}$ anymore. Actually, there are two types of vertices of $\mathcal{P}$:
$$
\begin{cases}
2n \text{ vectors of the form } (\pm 1, 0,\ldots,0) &\text{(type 1),}  \\
2^n \text{ vectors of the form } \left(\pm \frac{1}{2},\ldots,\pm \frac{1}{2}\right) & \text{(type 2).}
\end{cases}
$$

The non zero vectors of $S$ are contained in the hyperplane of $\R^n$ defined by the equation $x_{n-1}+x_n=1$. The faces of $\mathcal{P}$ are supported by the images of this hyperplane under the action of the group \textit{i.e.} the hyperplanes defined by the equations of the form $\pm x_i \pm x_j =1$, with $i\neq j$. Thus,
$$\begin{cases} & x\in \mathcal{P} \text{ if and only if for all } i\neq j,\  |x_i|+|x_j|\leq 1 \\ &  x\in \partial \mathcal{P} \text{ if and only if  } \max_{i\neq j}(|x_i|+|x_j|)=1, \end{cases}$$
and the norm $\Vert x\Vert_\mathcal{P}$ of a vector $x\in \R^n$ is 
$$ \Vert  x\Vert  _\mathcal{P} = \max_{i\neq j}(|x_i|+|x_j|) .$$

As in the case of $A_n$, the vertices of $\mathcal{P}$ span the dual lattice of $D_n$:

\begin{lemm}
The vertices of $\mathcal{P}$ span over $\Z$ the dual lattice  $ D_n^\# $.
\end{lemm}

\begin{proof}
It is immediate to check that for every $x\in D_n$ and $v\in V_\mathcal{P}$, $\langle x,v \rangle \in \Z$, so $\Span_\Z(V_\mathcal{P})\subset D_n^\#$. The converse follows directly from the following decomposition of $D_n^\#$:
$$ D_n^\# = D_n \cup \left(\left(\frac{1}{2},\ldots,\frac{1}{2}\right)+D_n\right) \cup \left(\left(\frac{1}{2},\ldots,-\frac{1}{2}\right)+D_n\right) \cup ((0,\ldots,0,1)+D_n).$$
\end{proof}

Once again, let $G_\mathcal{P}$ be the subgraph of $ G(\R^n,\Vert \cdot \Vert_\mathcal{P})$ induced by $V=\frac{1}{2} D_n^\#$, and let $\tilde{G}$ be the  auxiliary graph which is the  Cayley graph on $V$
associated with the generating set $\frac{1}{2}V_\mathcal{P}$. It also satisfies \eqref{Property D}:

\begin{lemm}\label{DP1DG2Dn}
The graph $\tilde{G}$ satisfies \eqref{Property D}.
\end{lemm} 

\begin{proof}
We follow the proof of \lref{DP1DG2An}. Let $v,v'\in V_\mathcal{P}$. We distinguish three cases depending on the type of $v$ and $v'$:
\begin{itemize}
\item If both $v$ and $v'$ are of type $1$, $v+v'$ is either $0$, or, up to permutation of the coordinates, of the form $(\pm 2,0,\ldots,0)$ or $(\pm 1, \pm 1, 0,\ldots,0)$, and ${\Vert  v+v'\Vert_\mathcal{P}  =2}$.
\item If both $v$ and $v'$ are of type $2$, the non zero coordinates of $v+v'$ are $1$ or $-1$. If $v+v'\neq 0$, then either it is a vertex of $V_\mathcal{P}$ of type $1$, or it has at least two coordinates whose absolute values are equal to $1$, and so $\Vert  v+v'\Vert_\mathcal{P}  =2$.
\item If $v$ is of type $1$ and $v'$ is of type $2$, then $v+v'$ is either a vertex of $V_\mathcal{P}$ of type $2$, or, up to a permutation of coordinates, of the form $\left(\pm \frac{3}{2},\pm \frac{1}{2},\ldots,\pm \frac{1}{2}\right)$, and $\Vert  v+v'\Vert_\mathcal{P}  =2$.
\end{itemize} 
\end{proof}

It remains to analyze the neighborhoods of the cliques of $\tilde{G}$. We first determine the possible cliques of $\tilde{G}$. We may assume that they contain $0$.

\begin{lemm} Up to symmetry, a clique of $\tilde{G}$ containing $0$ must be a subset of the maximal clique 
$$ C_{\max} =\left\{ 0, \frac{v_1}{2}, \frac{v_2}{2}, \frac{v_3}{2} \right\} \ \text{where } \begin{cases} v_1=(0,\ldots,0,1) \\ v_2=\left(\frac{1}{2},\frac{1}{2},\ldots,\frac{1}{2}\right) \\ v_3=\left(-\frac{1}{2},\frac{1}{2},\ldots,\frac{1}{2}\right)  \end{cases}.$$
\end{lemm}

\begin{proof}
Let $v,v'\in V_\mathcal{P}$ such that $\frac{v-v'}{2}\in \frac{1}{2}V_\mathcal{P}$. The conclusion follows from 
the following facts:
\begin{itemize}
\item Both $v$ and $v'$ can not be of type $1$, because the difference of two such vectors, is either $0$ or has polytope norm $2$.
\item If $v$ and $v'$ are both of type $2$, then $v$ and $v'$ must differ by only one coordinate, otherwise $\Vert  v-v'\Vert  _\mathcal{P}=2$.
\item If $v$ is of type $1$, say $v=(0,\cdots,0,\underbrace{\pm 1}_i,0\cdots,0)$, if $v'$ is of type $2$ and $\frac{v-v'}{2}\in \frac{1}{2}V_\mathcal{P}$, then the $i$th coordinate of $v'$ must have the same sign as the $i$th coordinate of $v$. 
\end{itemize}
\end{proof}

Then, we analyze  the local density of the  cliques:

\begin{lemm}\label{DnDelta*}
For every clique of $\tilde{G}$,
$$\delta^0(C)\leq \frac 1 {{(3/{4})2^n+n-1}}.$$
\end{lemm}

\begin{proof}
By enumerating the neighbors of every element in $C_{\max}$ and by counting the intersections of the different neighborhoods, we find that:

\begin{itemize}
\item If $C=\{0\}$, 
$\delta^0(C)=\frac{1}{1+2^n+n}$.
\item If $C=\left\{0,\frac {v_1} 2\right\}$, 
$$ \delta^0(C)=\frac{2}{2^n + 2^{n-1} +4n}=\frac{1}{(3/4)2^{n} +2n}. $$
Note that for $n\geq 6$, this density is already greater than $\frac{1}{2^n}$.
\item If $C$ is one of the two symmetric cliques $\left\{0,\frac{v_2} 2\right\}$ and $\left\{0,\frac{v_3} 2 \right\}$,
$$ \delta^0(C)=\frac{2}{ 2\times2^n + 2n}= \frac{1}{ 2^n + n}. $$
\item By symmetry, the cliques of the form $\left\{0,\frac{v_i} 2,\frac{v_j} 2\right\}$  have the same number of neighbors. If $C$ is one of them, 
$$ \delta^0(C)=\frac{3}{2\times 2^n + 2^{n-1} + 3n -1 } = \frac{1}{(5/6)2^n + n - 1/3 } ,$$

which is also greater than $\frac{1}{2^n}$.

\item Finally, 
$$ \delta^0(C_{\max})=\frac{4}{3\times 2^n + 4n - 4} = \frac{1}{{(3/{4})2^n+n-1}}, $$
which is the highest possible value of $\delta^0(C)$.
\end{itemize} 
\end{proof}

\section{The Chromatic number of $G(\R^n,\Vert\cdot\Vert_{\mathcal P})$}\label{SecChi}
In this section, we discuss the chromatic number 
$\chi(\R^n,\Vert \cdot \Vert_\mathcal{P})$ of the unit distance graph associated with a parallelohedron.
We start with the  construction of  a natural coloring of $\R^n$ with $2^n$ colors, leading to:
\begin{prop}
Let $\mathcal{P}$ be a parallelohedron in $\R^n$. Then 
$$ \chi(\R^n,\Vert \cdot \Vert_\mathcal{P})\leq 2^n. $$
\end{prop}
\begin{proof}
By assumption, there is a lattice $\Lambda$ such that $\R^n$ is the disjoint union $\cup_{\lambda \in \Lambda} (\lambda + \mathcal{P})$. We may also write $\R^n$ as the disjoint union 
$$\R^n=\bigcup_{\lambda \in \frac{1}{2}\Lambda} \left(\lambda + \frac{1}{2}\mathcal{P}\right)=\bigcup_{\lambda \in \frac{1}{2}\Lambda} B_\mathcal{P}\left(\lambda,\frac{1}{2}\right). $$
If $H$ is a coset of $\frac{1}{2} \Lambda\Big/ \Lambda$, then 
$$A_H=\bigcup_{\lambda \in H} B_\mathcal{P}\left(\lambda,\frac{1}{2}\right) $$
is a set avoiding distance $1$. So  the points in $A_H$ can receive the same color. This concludes the proof, since $\R^n$ is the disjoint union of all $A_H$ where $H$ runs through the $2^n$ cosets. 
\end{proof}

\begin{figure}[!ht]

\includegraphics[scale=0.25]{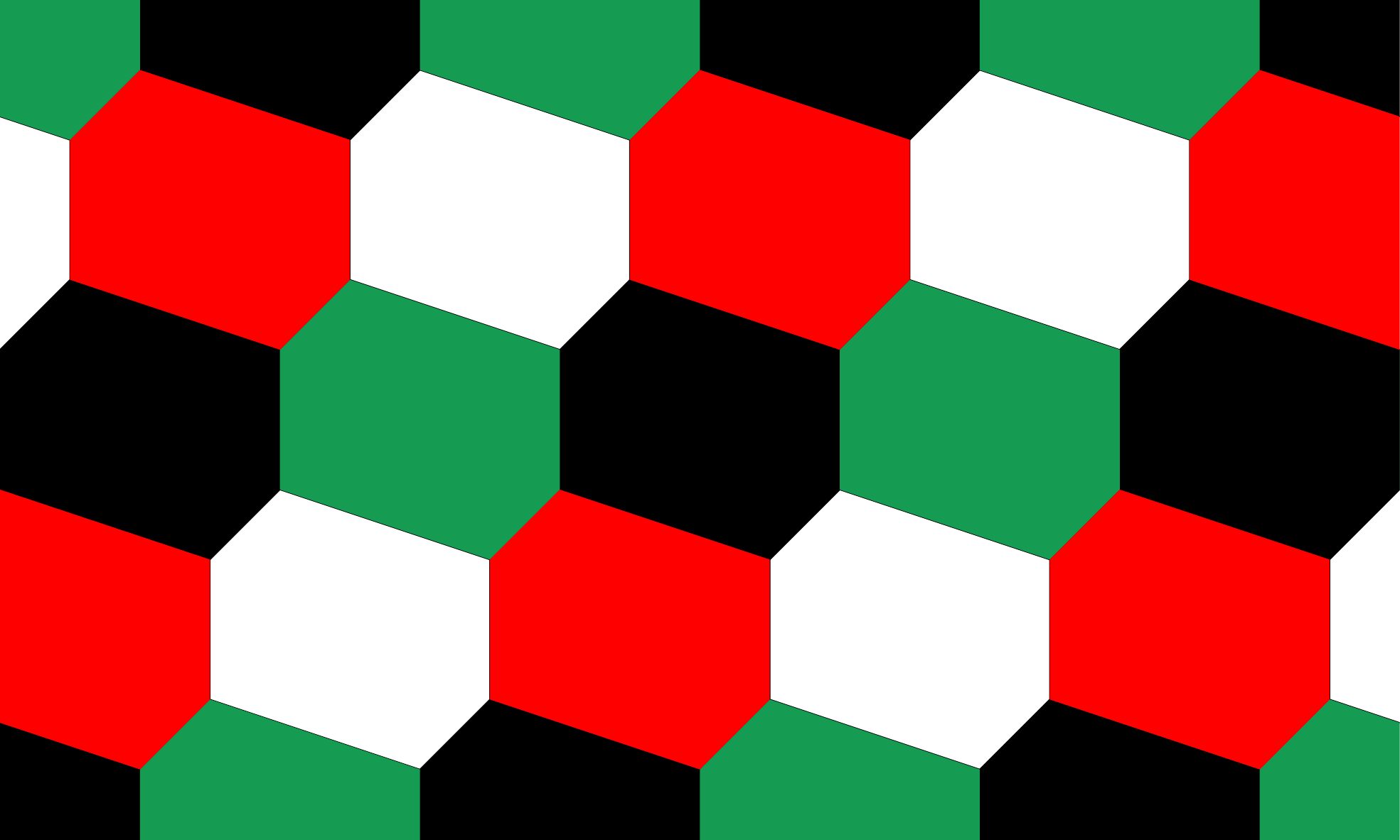}
\caption{$\chi(\R^n,\Vert \cdot \Vert_\mathcal{P})\leq 2^n $.}
\end{figure}

In order to lower bound  $\chi(\R^n,\Vert \cdot \Vert_\mathcal{P})$,  
we can take advantage of the induced subgraphs that we have constructed in previous sections.  
In particular, whenever we have a discrete induced subgraph $G_\mathcal{\mathcal{P}}$ of $G(\R^n,\Vert \cdot \Vert_\mathcal{P})$ satisfying $\bar{\alpha}(G_\mathcal{P})=\frac{1}{2^n}$, we obtain  as an immediate 
 consequence that
$$ \chi (\R^n,\Vert \cdot \Vert_\mathcal{P}) \geq \chi (G_\mathcal{P}) \geq \frac{1}{\bar{\alpha}(G_\mathcal{P})} =2^n. $$
Thus we have proved:
\begin{coro}
Let $\mathcal{P}$ be a parallelohedron in $\R^2$. Then
$$\chi (\R^2,\Vert \cdot \Vert_\mathcal{P})=4 .$$
\end{coro}

\begin{coro}
Let $\mathcal{P}$ be the Vorono\"\i\  region of the lattice $A_n$ in $\R^n$. Then
$$\chi (\R^n,\Vert \cdot \Vert_\mathcal{P})=2^n .$$
\end{coro}

\begin{rema}
We want to point out the fact that in dimension $2$, one can find a \emph{finite} induced subgraph of $G(\R^n,\Vert \cdot \Vert_\mathcal{P})$ with chromatic number $4$. Indeed,  the induced subgraph of $G(\R^n,\Vert \cdot \Vert_\mathcal{P})$ whose vertices and 
edges are drawn  in \fref{Chi} is easily seen to have chromatic number $4$.

\begin{figure}[!ht]
\includegraphics[scale=0.7]{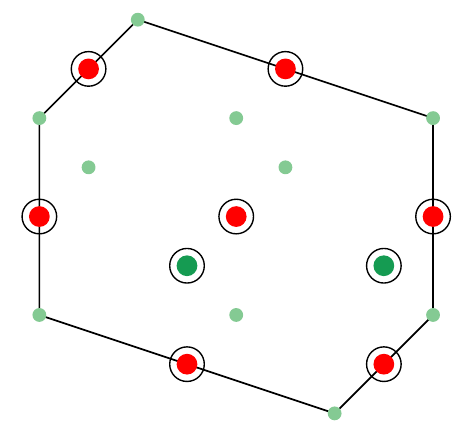}
\hspace{1cm}
\includegraphics[scale=0.7]{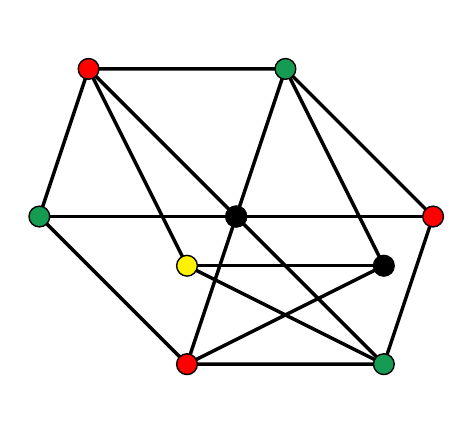}
\caption{$\chi(\R^2,\Vert \cdot \Vert_\mathcal{P})\geq 4$.\label{Chi}}
\end{figure}
\end{rema}

\appendix

\section{The proof of lemma \ref{alpha1=alpha2}}\label{AppA}

This appendix is dedicated to the proof of Lemma \ref{alpha1=alpha2}, which gives two equivalent formulations of the independence ratio 
of a discrete graph whose vertices have finite degrees. The importance of the assumption on the degrees of the vertices will be discussed after the proof. The statement of the lemma  is reproduced below:

\setcounter{lemm}{0}

\begin{lemm}
Let $G=(V,E)$ be a graph such that its vertex set $V$ is a discrete subset of $\R^n$, and such that every vertex  has finite degree. Then
$$ \bar{\alpha}(G)=\limsup_{R\to \infty} \bar{\alpha}(G_R),$$
where $G_R$ is the finite induced subgraph of $G$ whose set of vertices is $V_R=V\cap [-R,R]^n$.
\end{lemm}

\setcounter{lemm}{16}

\begin{proof} Let $G$ be a graph satisfying the assumptions of the lemma. First of all, we remark that
if $G$ is finite, there exists $R$ such that $G_R=G$. Thus, $\overline{\alpha}(G)= \limsup_{R\to\infty}\overline{\alpha} (G_R)$ is obvious.
From now on, we will assume that $G$ has infinitely many vertices.

The inequality $\overline{\alpha}(G)\leq \limsup_{R\to\infty}\overline{\alpha} (G_R)$ clearly holds. Indeed, if $A$ is an independent set of $G$, then $A\cap V_R$ is an independent set of $G_R$ and so $\frac{|A\cap V_R|}{|V_R|}\leq \overline{\alpha}(G_R)$, leading to
$\delta_G(A)\leq \limsup_{R\to\infty}\overline{\alpha}(G_R)$.

We will  prove the reverse inequality by exhibiting a sequence of independent sets $S_k$ such that, for all $k\geq 1$, $\limsup_{R\to\infty} \frac{|S_k\cap V_R|}{|V_R|}\geq \limsup_{R\to\infty}\overline{\alpha} (G_R) -\frac 1 k$.

Let $r_\ell$ be a strictly increasing sequence of real numbers tending to infinity and such that 
$\lim_{\ell\to\infty} \overline{\alpha}(G_{r_\ell})=\limsup_{R\to\infty}\overline{\alpha} (G_R)$, and let $A_{r_\ell}$ be an independent subset of $V_{r_\ell}$ of maximal cardinality. The set $S_k$ will be constructed from the
sequence of independent sets $A_{r_\ell}$; however, we will need, for reasons that will appear more clearly later, that the successive rings $V_{r_\ell}\setminus V_{r_{\ell-1}}$ are sufficiently large. In view of that, we construct a convenient subsequence 
of $r_\ell$, with the help of a function $\varphi(\ell)$, in the following way. 

Since the graph $G$ is discrete, we know that for all $R$, $V_R$ is finite and since all the vertices of the graph are of finite degree, we know that the neighborhood $N[V_R]$ is finite too. We call $b(R)$ the smallest real number such that $N[V_R]\subset V_{b(R)}$. Then, we set $\varphi(0)=0$ and, inductively for $\ell\geq 0$,
\[
\varphi(\ell+1)=\min\left\{i\mid r_{i}\geq b(r_{\varphi(\ell)})\text{ and }|V_{r_{i}}\setminus V_{r_{\varphi(\ell)}}|\geq |V_{r_{\varphi(\ell)}}\setminus V_{r_{\varphi(\ell-1)}}|\right\}.\]

The existence of $\varphi(\ell+1)$ at each step of the recursion holds because $\lim_{\ell\to\infty} r_\ell=+\infty$ and  $V_{r_{\varphi(\ell)}}\setminus V_{r_{\varphi(\ell-1)}}$ is finite (since $G$ is discrete). To keep the notations simple, we set $R_\ell=r_{\varphi(\ell)}$.

We will need the following property of the number of elements of the rings associated to the sequence $R_\ell$:

\begin{prop}
For all $\ell\in \mathbb N$, for all $m\in\mathbb N^{\ast}$:

\[ |V_{R_{\ell+1}}\setminus V_{R_{\ell}}|\leq \frac 1 m |V_{R_{\ell+m}}\setminus V_{R_{\ell}}|\]
\label{1i}
\end{prop}

\begin{proof}
We have $|V_{R_{\ell+m}}\setminus V_{R_{\ell}}|=\underset{k=0}{\overset {m-1} {\sum}} |V_{R_{\ell+k+1}}\setminus V_{R_{\ell+k}}|$ and each term of the sum is larger than $|V_{R_{\ell+1}}\setminus V_{R_{\ell}}|$, by definition of $\varphi$.
\end{proof}

Now we are ready to define the sets $S_k$. We set, for $k\geq 0$,

\[S_k:=\left\{v\in V\mid \exists i\in\mathbb N\text{ such that }v\in A_{R_{ik}}\text{ and }\forall j<i, v\notin N[A_{R_{jk}}]
\right\}.\]

It remains to prove that $S_k$ is an independent set and satisfies the inequality  $\limsup_{R\to+\infty} \frac{|S_k\cap V_R|}{|V_R|}\geq \limsup_{R\to\infty}\overline{\alpha} (G_R) -\frac 1 k$.
\vspace{0.35cm}
\begin{prop} 
$S_k$ is independent. 
\end{prop}

\begin{proof}
Let $v_1$ and $v_2$ be two vertices of $S_k$ and let $i_1$ and $i_2$ be such that $v_1\in A_{R_{i_1k}}$, $v_2\in A_{R_{i_2k}}$ and for all $j<i_1$ (respectively $i_2$), $v_1$ (resp. $v_2$) $\notin N[A_{R_{jn}}]$. 
If $i_1=i_2$, then $v_1$ and $v_2$ both belong to $A_{R_{i_1k}}$ which is independent, consequently they are not connected. If say, $i_1>i_2$, from the very definition of $S_k$, $v_2\notin N[A_{R_{i_1k}}]$, so $v_1$ and $v_2$ are not connected either.
\end{proof}

\begin{lemm}\label{induc}
For all $k\geq 1$, $i\geq 0$, $\frac{|S_k\cap V_{R_{ik}}|}{|V_{R_{ik}}|}\geq \frac{|A_{R_{ik}}|}{|V_{R_{ik}}|}-\frac 1 k$.
\end{lemm}

\begin{proof}
 We prove the lemma by induction on $i$:
 
 $\bullet$ The property holds for $i=0$ since $S_k\cap V_{R_0}$ contains $A_{R_0}$.
 
 $\bullet$ Let $i\in\mathbb N$ be such that the property holds. We have:
\begin{equation}\label{eq1}
\frac{|S_k\cap V_{R_{(i+1)k}}|}{|V_{R_{(i+1)k}}|}=\frac{|S_k\cap V_{R_{ik}}|}{|V_{R_{(i+1)k}}|}+\frac{|S_k\cap (V_{R_{(i+1)k}}\setminus V_{R_{ik}})|}{|V_{R_{(i+1)k}}|}.
\end{equation}

Let us lower bound the two terms of this sum one after the other.

$\blacktriangleright$ Since $A_{R_{ik}}$ is an independent set of maximal cardinality in $V_{R_{ik}}$, we know that  
$|A_{R_{ik}}|\geq |A_{R_{(i+1)k}}\cap V_{R_{ik}}|$. Combining with the induction hypothesis, we find
\[\frac{|S_k\cap V_{R_{ik}}|}{|V_{R_{ik}}|}\geq\frac{|A_{R_{(i+1)k}}\cap V_{R_{ik}}|}{|V_{R_{ik}}|}-\frac 1 k\]
and thus:
\begin{equation}\label{eq2}\frac{|S_k\cap V_{R_{ik}}|}{|V_{R_{(i+1)k}}|}\geq\frac{|A_{R_{(i+1)k}}\cap V_{R_{ik}}|}{|V_{R_{(i+1)k}}|}-\frac 1 k \frac {|V_{R_{ik}}|} {|V_{R_{(i+1)k}}|}.\end{equation}

$\blacktriangleright$ By definition, $S_k$ contains all the vertices of $A_{R_{(i+1)k}}$ except those who are in the neighborhood of an $A_{R_{jk}}$ with $j<i+1$. Since for all $j<i$, $N[A_{R_{jk}}]\subset V_{b(R_{ik})}$, the set 
$S_k\cap (V_{R_{(i+1)k}}\setminus V_{R_{ik}})$ contains  $A_{R_{(i+1)k}}\setminus V_{b(R_{ik})}$.
We also have by construction that $b(R_{ik})\leq R_{ik+1}$. 
Thus, \[|V_{b(R_{ik})}\setminus V_{R_{ik}}|\leq |V_{R_{ik+1}}\setminus V_{R_{ik}}|\leq \frac 1 k |V_{R_{(i+1)k}}\setminus V_{R_{ik}}|\]
where the second inequality follows from Proposition \ref{1i}.

This leads to the following inequality:
\begin{equation}\label{eq3}\frac{|S_k\cap (V_{R_{(i+1)k}}\setminus V_{R_{ik}})|}{|V_{R_{(i+1)k}}|}\geq \frac{|A_{R_{(i+1)k}}\setminus V_{R_{ik}}|}{|V_{R_{(i+1)k}}|}-\frac 1 k \frac{|V_{R_{(i+1)k}}\setminus V_{R_{ik}}|}{|V_{R_{(i+1)k}}|}\end{equation}

By combining equations (\ref{eq1}), (\ref{eq2}) and (\ref{eq3}), we find:
\[\frac{|S_k\cap V_{R_{(i+1)k}}|}{|V_{R_{(i+1)k}}|}\geq \frac{|A_{R_{(i+1)k}}|}{|V_{R_{(i+1)k}}|}-\frac 1 k\]
which concludes the proof of  Lemma \ref{induc}.
\end{proof}


Now we are ready to conclude the proof of Lemma \ref{alpha1=alpha2}. Indeed, for all $k\geq 1$, we have

\[\begin{split}
\overline{\alpha}(G)=\sup_S \quad \limsup_{R\to+\infty} \frac{|S\cap V_R|}{|V_R|}&\geq \limsup_{R\to+\infty} \frac{|S_n\cap V_R|}{|V_R|}\\
&\geq \lim_{i\to\infty} \frac{|S_k\cap V_{R_{ik}}|}{|V_{R_{ik}}|}\\
&\geq \lim_{i\to\infty}\frac{|A_{R_{ik}}|}{|V_{R_{ik}}|}-\frac 1 k\\
&\geq \limsup_{R\to\infty}\overline{\alpha} (G_R) -\frac 1 k.
  \end{split}\]

In the limit when $k\to \infty$, we obtain that $\overline{\alpha}(G)\geq \limsup_{R\to\infty}\overline{\alpha} (G_R)$. 
\end{proof}

To conclude our discussion of Lemma \ref{alpha1=alpha2}, we would like to point out that the inequality $\overline{\alpha}(G)\leq \limsup_{R\to\infty}\overline{\alpha} (G_R)$ does not necessarily hold if $G$ has vertices with infinite degree, by bringing out a counterexample.

Let $G$ be the graph given by $V=\mathbb Z$ and $E=\{\{a,b\}|a<0\text{ and }b> -2a\}$.

Let $N\in\mathbb N$ and let $S_N=\left[\!\!\left[ -N,-\frac N 2\right]\!\!\right] \cup[\![ 0,N]\!]$. One can see 
easily that $S_N$ is independent in $G$. Hence, $\limsup_{R\to\infty}\overline{\alpha} (G_R)\geq\lim_{N\to\infty}\frac{|S_N|}{|V_N|}=\frac 3 4 $.

Let $S$ be an independent set of $G$. If $S$ contains a vertex indexed by a negative integer $-k$, it cannot contain any vertex indexed by $i>2k$ and can therefore only contain finitely many vertices indexed by positive integers. Hence, $\lim_{N\to\infty}\frac{|S\cap V_N|}{|V_N|}\leq \frac 1 2$. If $S$ does not contain any vertex indexed by a negative integer, the inequality $\lim_{N\to\infty}\frac{|A\cap V_N|}{|V_N|}\leq \frac 1 2$ holds aswell. Thus, \[\sup_S \limsup_{N\to+\infty} \frac{|S\cap V_N|}{|V_N|}\leq \frac 1 2\]
which proves that $\overline{\alpha}(G)\neq \limsup_{R\to\infty}\overline{\alpha} (G_R)$.

\section*{Acknowledgements}
We would like to thank Sinai Robins for suggesting the polytopal norms as an alternative to the standard Euclidean norm in the study of $1$-avoiding sets, and for helpful discussions on this subject.

\bibliographystyle{plain}
\bibliography{Biblio}

\begin{thebibliography}{10}

\bibitem{MR3341578}
Christine Bachoc, Alberto Passuello, and Alain Thiery.
\newblock The density of sets avoiding distance 1 in {E}uclidean space.
\newblock {\em Discrete Comput. Geom.}, 53(4):783--808, 2015.

\bibitem{Conway:1987:SLG:39091}
J.~H. Conway and N.~J.~A. Sloane.
\newblock {\em Sphere-packings, Lattices, and Groups}.
\newblock Springer-Verlag New York, Inc., New York, NY, USA, 1987.

\bibitem{Croft}
H.~T. Croft.
\newblock Incidence incidents.
\newblock {\em Eureka (Cambridge)}, 30:22--26, 1967.

\bibitem{zbMATH02567417}
B.~{Delaunay}.
\newblock {Sur la partition r\'eguli\`ere de l'espace \`a 4 dimensions. I, II.}
\newblock {\em {Bull. Acad. Sci. URSS}}, 2:79--110, 1929.

\bibitem{MR1703597}
R.~M. Erdahl.
\newblock Zonotopes, dicings, and {V}oronoi's conjecture on parallelohedra.
\newblock {\em European J. Combin.}, 20(6):527--549, 1999.

\bibitem{Falconer}
K.~J. Falconer.
\newblock The realization of distances in measurable subsets covering rn.
\newblock {\em Journal of Combinatorial Theory, Series A}, 31(2):184 -- 189,
  1981.

\bibitem{Keleti2016}
Tam{\'a}s Keleti, M{\'a}t{\'e} Matolcsi, Fernando~M{\'a}rio de~Oliveira~Filho,
  and Imre~Z. Ruzsa.
\newblock Better bounds for planar sets avoiding unit distances.
\newblock {\em Discrete Comput. Geom.}, 55(3):642--661, 2016.

\bibitem{MR0319055}
D.~G. Larman and C.~A. Rogers.
\newblock The realization of distances within sets in {E}uclidean space.
\newblock {\em Mathematika}, 19:1--24, 1972.

\bibitem{MR582003}
P.~McMullen.
\newblock Convex bodies which tile space by translation.
\newblock {\em Mathematika}, 27(1):113--121, 1980.

\bibitem{Minkowski1897}
Hermann Minkowski.
\newblock Allgemeine lehrsätze über die convexen polyeder.
\newblock {\em Nachrichten von der Gesellschaft der Wissenschaften zu
  Göttingen, Mathematisch-Physikalische Klasse}, 1897:198--220, 1897.

\bibitem{Dmitry}
Dmitry Shiryaev.
\newblock Personal communication.

\bibitem{Soifer}
Alexander Soifer.
\newblock {\em The mathematical coloring book: Mathematics of coloring and the
  colorful life of its creators}.
\newblock Springer Science \& Business Media, 2008.

\bibitem{MR1954746}
L.~A. Sz{\'e}kely.
\newblock Erd{\H o}s on unit distances and the {S}zemer\'edi-{T}rotter
  theorems.
\newblock In {\em Paul {E}rd{\H o}s and his mathematics, {II} ({B}udapest,
  1999)}, volume~11 of {\em Bolyai Soc. Math. Stud.}, pages 649--666. J\'anos
  Bolyai Math. Soc., Budapest, 2002.

\bibitem{MR0094790}
B.~A. Venkov.
\newblock On a class of {E}uclidean polyhedra.
\newblock {\em Vestnik Leningrad. Univ. Ser. Mat. Fiz. Him.}, 9(2):11--31,
  1954.

\bibitem{MR1580754}
Georges Voronoi.
\newblock Nouvelles applications des param\`etres continus \`a la th\'eorie des
  formes quadratiques. {D}euxi\`eme m\'emoire. {R}echerches sur les
  parall\'ello\`edres primitifs.
\newblock {\em J. Reine Angew. Math.}, 134:198--287, 1908.

\end{thebibliography}

\end{document}